\documentclass[11pt]{article}
\usepackage[top=1in, bottom=1in, left=1in, right=1in]{geometry}
\usepackage{paper}

\usepackage{todonotes} 
\usepackage{comment}
\usepackage{bbm}
\usepackage{empheq}

\renewcommand{\trace}[1]{\mathrm{tr}\left(#1\right)}
\newcommand{\triu}[1]{\mathrm{triu}\left(#1\right)}

\newcommand{\ie}{\latin{i.e.}}
\newcommand{\eg}{\latin{e.g.}}

\newcommand{\tx}{\widetilde{x}}

\newcommand{\teps}{\widetilde{\varepsilon}}

\newcommand{\veps}{\varepsilon}

\newcommand{\tA}{\widetilde{A}}

\newcommand{\OFM}{OFM}
\newcommand{\TOM}{TriOFM}
\newcommand{\OFMOne}{OFM-(Obj1)}
\newcommand{\TOMOne}{TriOFM-(Obj1)}
\newcommand{\TOMTwo}{TriOFM-(Obj2)}

\begin{document}

\title{Triangularized Orthogonalization-free Method for Solving
Extreme Eigenvalue Problems}

\author{
    Weiguo Gao$^{\,\dagger\,\ddagger\,\mathsection}$,
    Yingzhou Li$^{\,\dagger}$,
    Bichen Lu$^{\,\mathsection\,\dagger}$
    \vspace{0.1in}\\
    $\dagger$ School of Mathematical Sciences, Fudan University\\
    $\ddagger$ School of Data Science, Fudan University\\
    $\mathsection$ Shanghai Center for Mathematical Sciences\\
    {\let\thefootnote\relax\footnote{{Authors are listed in alphabetical order.}}}
}

\maketitle

\begin{abstract}
A novel orthogonalization-free method together with two specific
algorithms are proposed to address extreme eigenvalue problems. On top of
gradient-based algorithms, the proposed algorithms modify the multicolumn
gradient such that earlier columns are decoupled from later ones. Locally,
both algorithms converge linearly with convergence rates depending on
eigengaps. Momentum acceleration, exact linesearch, and column locking are
incorporated to accelerate algorithms and reduce their computational
costs. We demonstrate the efficiency of both algorithms on random matrices
with different spectrum distributions and matrices from computational
chemistry.
\end{abstract}


\reversemarginpar

\section{Introduction}
\label{sec:introduction}

This paper proposes a novel triangularized orthogonalization-free
method (\TOM{}) for solving extreme eigenvalue problems. Given a
symmetric matrix $A$, the extreme eigenvalue problem is defined as,
\begin{equation} \label{eq:EVP}
    AU = U \Lambda,
\end{equation}
where $A \in \bbR^{n \times n}$, $A^\top = A$, $\Lambda \in \bbR^{p \times
p}$ is a diagonal matrix with $A$'s $p$ smallest eigenvalues on the
diagonal in ascending order, and the columns of $U$ are the corresponding
eigenvectors. The proposed methods target some specific applications in
computational chemistry, in which areas smallest eigenpairs are desired as
the ground-state and low-lying excited-states. Though we introduce
algorithms for $p$ smallest eigenpairs, all algorithms in this paper can
be adapted to compute $p$ largest eigenpairs. Besides computational
chemistry, solving extreme eigenvalue problems is a fundamental
computational step in a wide range of applications, including but not
limited to the principal component analysis, dimension reduction, spectral
clustering, etc.

In this paper, we specifically concern extreme eigenvalue problems
with two properties:
\begin{enumerate}[(i)]
    \item Orthogonalization of the iteration variable $X$ is not permitted;
    \item Eigenvectors are sparse vectors.
\end{enumerate}
At least two important applications from computational chemistry,
linear-scaling density functional theory (DFT)~\cite{Mauri1993} and full
configuration interaction (FCI)~\cite{Knowles1984} for low-lying excited
states, admit these two properties. In linear-scaling DFT, the number of
desired eigenpairs is of the same order as the problem size. The
orthogonalization step then scales cubically, which is not permitted in
linear-scaling DFT. Regarding the sparsity, linear-scaling DFT adopts
localzed basis sets, and the eigenvectors therein are indeed
sparse~\cite{Brouder2007, Stubbs2020}. Although FCI also admits the above
two properties, it has its own unique feature. In FCI, the desired number
of eigenpairs $p$ is usually a small constant, e.g., $p=5,10$. While the
size of the matrix $n$ grows factorially as the system increases. For
example, considering a single water molecule with 48 spin-orbitals and 10
electrons, the matrix is of size $\sim 10^8$. Due to the factorially
increasing matrix size, orthogonalization is too expensive in both
computational and memory costs to be applied in practice. Sparsity is also
an important feature of FCI. Thanks to the two-body interaction feature of
the electrons, the matrix is extremely sparse. Regarding the water
molecule example, each column of the matrix has roughly $10^4$ nonzero
entries. The eigenvectors of the ground-state and low-lying excited-states
are sparse. FCI is the motivating application of this work, and hence some
of our algorithm designs would prefer FCI to DFT.

\subsection{Related Work}

For linear symmetric eigenvalue problems as \eqref{eq:EVP}, there are many
classical eigensolvers from textbooks of numerical linear algebra. Readers
are referred to~\cite{Golub2013} for references. In electronic structure
calculation, variants of classical eigensolvers, like
Davidson~\cite{Davidson1975}, locally optimal block preconditioned
conjugate gradient method (LOBPCG)~\cite{Knyazev2001}, projected
preconditioned conjugate gradient (PPCG)~\cite{Vecharynski2015}, Chebyshev
filtering~\cite{Banerjee2016, Zhou2006}, pole expansion~\cite{Li2017b,
PeterTang2014}, are widely used in the self-consistent field iteration in
DFT. All these methods are related to Krylov subspace. A recent software
ELSI~\cite{Yu2018, Yu2019} provides an interface to many of these
eigensolvers for DFT calculation.

Besides Krylov subspace methods, another family of methods view the
symmetric eigenvalue problem as a constrained optimization problem and
solve it using either first-order or second-order methods~\cite{Dai2019b,
Gao2018, Huang2015, Wen2013, Zhang2014}. These methods usually target more
general objective functions with orthonormal constraints. However, the
linear eigenvalue problem is always one of their important applications.
Since the feasible set of the orthonormality constraint is the Stiefel
manifold, these methods are also known as manifold optimization methods.
They take either the Euclidean gradient or Riemannian gradient step with
certain strategies in calculating the stepsize. A retraction or projection
step is needed to maintain the feasibility of the iteration variable.
Recently, in order to enhance the parallelizability, the retraction step
is avoided through either the augmented Lagrangian method~\cite{Gao2019,
Wen2016} or extend gradient~\cite{Dai2019a}.

Linear symmetric eigenvalue problems can also be written as an
unconstrained optimization problem. The most well-known one is
minimizing the Rayleigh quotient, which can be generalized to the
multicolumn case. Another two unconstrained optimization problems are
\begin{equation} \label{eq:obj1} \tag{Obj1}
    \min_{X\in \mathbb{R}^{n\times p}} \fnorm{A + XX^\top}^2,
\end{equation}
and
\begin{equation} \label{eq:obj2} \tag{Obj2}
    \min_{X\in \mathbb{R}^{n\times p}} \trace{(2I-X^\top X)X^\top AX},
\end{equation}
where $\fnorm{\cdot}$ denotes the Frobenius norm and $\trace{\cdot}$
denotes the trace operation. \eqref{eq:obj1} has been adopted
to address the extreme eigenvalue problems arsing from several
areas~\cite{Lei2016, Li2019c, Liu2015c}, including FCI~\cite{Li2020a,
Wang2019}. \eqref{eq:obj2} is widely known as the orbital minimization
method (OMM)~\cite{Corsetti2014, Lu2017a, Lu2017, Ordejon1993,
Mauri1993}, which is popular in the area of (linear-scaling) DFT. More
details about \eqref{eq:obj1} and \eqref{eq:obj2} are deferred to
Section~\ref{sec:preliminary}.

For all methods aforementioned in this section, some of them are
orthogonalization-free, and some of them converge to eigenvectors
directly.  Nevertheless, none of them is an orthogonalization-free method
converging to eigenvectors directly.

\subsection{Contribution}

In this paper, a novel iterative method named triangularized
orthogonalization-free method~(\TOM{}) is proposed, which is
orthogonalization-free and converges to eigenvectors directly. The method
is inspired by the unconstrained optimization methods (denoted as \OFM{}
throughout this paper). In \OFM{}, the updating direction is the gradient
of the objective function, whereas, in \TOM{}, the updating direction is a
triangularized version of the gradient, which decouples earlier columns
from later ones. When the gradient is triangularized in \TOM{}, the
updating direction is no longer a gradient of any function. Hence the
underlying dynamic is not a conservative flow. The analysis is then very
different from traditional analysis in optimization. In this paper, we
triangularize two objective functions, i.e., \eqref{eq:obj1} and
\eqref{eq:obj2}, and obtain two iterative algorithms named \TOMOne{} and
\TOMTwo{} respectively.

The convergence analysis of \TOMOne{} is carried out in detail. First,
we discuss the stable and unstable fixed points of our algorithm.
We then provide local convergence analysis with a convergence
rate. The rate is carried out through a careful analysis of the
accumulated error term. All analyses can be extended to \TOMTwo{},
and we state the corresponding theorems without detailed proof.
Global convergence can be also be established.  We leave the
detail in a companion paper~\cite{Gao2021}. Notice that the global
convergence is given without a rate.

After the analyses, we propose a few techniques to accelerate the
convergence and reduce the computational cost. Conjugate gradient
direction and linesearch strategies are proposed to accelerate both
algorithms. These two techniques were also applied in \OFM{} which are
tailored for \TOM{} in this paper. While, in \OFM{}, the locking technique
is not feasible due to the existence of the orthogonalization step. In
\TOM{}, the locking technique is incorporated to reduce the computational
cost.

Finally, numerical examples are provided to demonstrate the
effectiveness of \TOM{}. All suggested techniques are first explored on
random matrices and then applied to two practical examples, one from
DFT and another one from FCI. In both practical examples, we observe
that the proposed framework achieves both the orthogonalization-free
and converging to eigenvectors properties while not losing much
efficiency comparing with their original \OFM{} counterparts.

\subsection{Organization}

In the rest of this paper, Section~\ref{sec:preliminary} provides
detailed introductions to both \eqref{eq:obj1} and \eqref{eq:obj2}
with an analysis of the energy landscape. Section~\ref{sec:Triu-opt}
introduces \TOM{} and its two iterative algorithms, \TOMOne{} and
\TOMTwo{}, in detail. The convergence analysis is carried out in
Section~\ref{sec:localconv}. Algorithmic techniques are proposed in
Section~\ref{sec:implementation}. In Section~\ref{sec:numerical}, all
algorithms are numerically explored on random matrices and matrices
from practice.  Finally, Section~\ref{sec:conclusion} concludes the
paper with a discussion on future directions.

\section{Preliminary}
\label{sec:preliminary}

We introduce \OFM{} eigensolvers based on \eqref{eq:obj1} and
\eqref{eq:obj2} in this section. Notations used throughout the paper
are summarized in Table~\ref{tab:notations}, which would be used
without further explanation.

\begin{table}[ht]  
    \centering
    \begin{tabular}{lp{0.7\textwidth}}
        \toprule  
        Notation & Explanation \\ 
        \toprule
        $n$ & The size of the matrix.\\
        $q$ & The number of negative eigenvalues of the matrix.\\
        $p$ & The number of desired eigenpairs and $p \leq q$.\\
        \midrule
        $A$ & The $n$-by-$n$ symmetric matrix.\\
        $\Lambda$ & A diagonal matrix with diagonal entries being
        eigenvalues of $A$ in increasing ordering.\\
        $\lambda_i$ & The $i$-th smallest eigenvalue of $A$.\\
        $\Lambda_i$ & The first $i$-by-$i$ principal submatrix of
        $\Lambda$.\\
        $U$ & An orthogonal matrix satisfying $U^\top A U=\Lambda$.\\
        $u_i$ & The eigenvector of $A$ corresponding to $\lambda_i$.\\
        $U_i$ & The first $i$ columns of $U$.\\
        $\rho$ & The 2-norm of $A$, \ie, $\rho = \norm{A}_2$.\\
        \midrule
        $X^{(t)}$ & An $n$-by-$p$ matrix denoting the iteration variable at
        $t$-th iteration.\\
        $x^{(t)}_i$ & The $i$-th column of $X^{(t)}$.\\
        $X^{(t)}_i$ & The first $i$ columns of $X^{(t)}$.\\
        $f_1(X), f_2(X)$ & The objective function in \eqref{eq:obj1},
        \eqref{eq:obj2}.\\
        $\grad{f_1}(X), \grad{f_2}(X)$ & The gradient of
        $f_1(X)$, $f_2(X)$.\\
        \midrule
        $\alpha$ & The stepsize.\\
        $e_i$ & The $i$-th standard basis vector~\footnotemark.\\
        \bottomrule
    \end{tabular}
    \caption{Notations.}
    \label{tab:notations}
\end{table}
\footnotetext[1]{A vector of length $n$ with one on the $i$-th entry
and zero elsewhere.}

The orthogonalization step is a key step in most traditional eigensolvers,
\eg, power method, QR iteration, Lanczos, etc.  While, the
orthogonalization step is difficult to be efficiently parallelized on
modern computer architectures, \ie, distributed-memory computers and GPUs.
\OFM{} eigensolvers, in contrast, do not involve the orthogonalization
step and only require matrix-matrix multiplication, which is one of the
most parallel efficient operations. Hence \OFM{} is plausible for solving
large problems in a massively parallel environment.

Given that $A$ is symmetric, $A + X X^\top$ in \eqref{eq:obj1} is the
residual of a symmetric low-rank approximation. In \cite{Liu2015c},
\eqref{eq:obj1} is shown to be equivalent to a trace-penalty
minimization model with a specific penalty parameter.  In both
\cite{Li2019c, Liu2015c}, the energy landscape of \eqref{eq:obj1}
has been analyzed. \eqref{eq:obj1} does not have any spurious local
minimum, and all local minima are global minima. We rephrase and
summarize the analysis result as follows.

\begin{theorem} \label{thm:obj1}
    All {\bfseries stationary points} of \eqref{eq:obj1} are of form $X =
    U_q \sqrt{-\Lambda_q} S P$ and all {\bfseries local minima} are of form
    $X = U_p \sqrt{-\Lambda_p} Q$, $S \in \bbR^{q \times q}$ is a diagonal
    matrix with diagonal entries being 0 or 1 (at most $p$ 1s), $P \in
    \bbR^{q\times p}$ and $Q \in \bbR^{p\times p}$ are unitary matrices.
    Further, any local minimum is also a global minimum.
\end{theorem}

Notice that in Theorem~\ref{thm:obj1}, $A$ is implicitly assumed
to have at least $p$ negative eigenvalues. On the other hand, when
$p > q$, stationary points are exactly of the same form, whereas
local minima need to be updated as $X = U_q \sqrt{-\Lambda_q} Q$
with $Q \in \bbR^{q\times p}$ having orthogonal columns. For almost all
chemistry problems, the assumption $p \leq q$ holds in practice.
Hence we stick to this assumption for \eqref{eq:obj1} throughout the
paper to simplify our presentation.

The intuition behind \eqref{eq:obj2} is more complicated. There are two
ways to motivate the objective function: the approximated inverse and the
Lagrange multiplier.

A multi-column version of Rayleigh quotient admits $\trace{(X^\top
X)^{-1} X^\top A X}$, which could also be an objective function option
for \OFM{}. Assuming the spectrum of $X^\top X$ is bounded by one,
we have the Neumann series expansion of the inversion and the first
order approximation as,
\begin{equation*}
    (X^\top X)^{-1} = (I - (I - X^\top X) )^{-1} = \sum_{k=0}^{\infty}
    (I - X^\top X )^k \approx 2I - X^\top X.
\end{equation*}
Substituting the approximation into the multi-column Rayleigh quotient
leads to \eqref{eq:obj2}.

Another way to motivate \eqref{eq:obj2} is via the Lagrange multiplier
method. Lagrangian function for eigenvalue problem admits
\begin{equation*}
    \calL (X, \Xi ) = \trace{X^\top A X} - \trace{\Xi (X^\top X - I)},
\end{equation*}
where $\Xi$ denotes the Lagrange multiplier. The first order optimality
condition leads to an expression for the Lagrange multiplier, $\Xi =
X^\top A X$. Substituting this expression into the Lagrangian function
gives \eqref{eq:obj2}.

Previous work~\cite{Lu2017a} characterizes the energy landscape of
\eqref{eq:obj2}.  \eqref{eq:obj2} does not have any spurious local
minimum. The theorem therein is rephrased as follows.

\begin{theorem} \label{thm:obj2}
    Let $A$ be a symmetric negative semi-definite matrix. All {\bfseries
    stationary points} of \eqref{eq:obj2} are of form $X = U S P$ and all
    {\bfseries local minima} are of form $X = U_p Q$, where $S \in \bbR^{n
    \times n}$ is a diagonal matrix with diagonal entries being 0 or 1 (at
    most $p$ 1s), $P \in \bbR^{n\times p}$ and $Q \in \bbR^{p\times p}$ are
    unitary matrices. Further, any local minimum is also a global minimum.
\end{theorem}

Notice that the matrix $A$ in \eqref{eq:obj2} must be negative
semi-definite. Otherwise, $X$ can be scaled eigenvectors corresponding to
the positive eigenvalues, and \eqref{eq:obj2} is unbounded from below. For
eigenvalue problems, the matrix can be shifted to be negative
semi-definite. Comparing to \eqref{eq:obj1}, an extra step is needed to
estimate the shift, and shifting is needed every iteration.

Based on the analysis of the energy landscape of both \eqref{eq:obj1} and
\eqref{eq:obj2}, any algorithm avoiding saddle points converges to the
global minimum. Such algorithms include but are not limited to regular
gradient descent~\cite{Lee2019}, conjugate gradient descent, stochastic
gradient descent~\cite{Bottou2018, Li2019}, etc. Using the notation
defined in Table~\ref{tab:notations}, gradients of \eqref{eq:obj1} and
\eqref{eq:obj2} are
\begin{equation} \label{eq:grad-obj1}
    \grad{f_1}(X) = 4AX + 4XX^\top X,
\end{equation}
and
\begin{equation} \label{eq:grad-obj2}
    \grad{f_2}(X) = 4AX - 2XX^\top AX - 2AXX^\top X,
\end{equation}
respectively. The gradient descent iterations are defined as,
\begin{equation} \label{eq:iterobj1}
    X^{(t+1)} = X^{(t)} - \alpha \left( AX^{(t)} + X^{(t)} \left(
    X^{(t)} \right)^\top X^{(t)} \right),
\end{equation}
and
\begin{equation} \label{eq:iterobj2}
    X^{(t+1)} = X^{(t)} - \alpha \left( 2AX^{(t)} - AX^{(t)} \left(
    X^{(t)} \right)^\top X^{(t)} - X^{(t)} \left( X^{(t)} \right)^\top
    AX^{(t)} \right),
\end{equation}
where the constant is absorbed into the stepsize. Unfortunately, the
Hessian of both \eqref{eq:obj1} and \eqref{eq:obj2} are unbounded from
above. The valid set for the choice of the stepsize over the entire domain
is empty. For both \eqref{eq:obj1} and \eqref{eq:obj2}, one can find a
bounded domain such that iterations are guaranteed to stay within the
domain. Then Hessians are bounded over the domain and the valid set for
the stepsize is non-empty.

\section{Triangularized Optimization Eigensolvers}
\label{sec:Triu-opt}

We propose triangularized orthogonalization-free methods (\TOM{}) as
eigensolvers based on \eqref{eq:obj1} and \eqref{eq:obj2}, which are
denoted as \TOMOne{} and \TOMTwo{}.

Our goal, as mentioned in Section~\ref{sec:introduction} is to find $p$
extreme eigenpairs with two properties: (i). orthogonalization of $X$ is
not permitted; (ii). eigenvectors are sparse vectors. Optimizing
\eqref{eq:obj1} and \eqref{eq:obj2} almost achieves the first required
property except for the post-processing part, while the second property is
not taken into consideration. Due to the existence of the arbitrary
orthogonal matrix $Q$, the iterations \eqref{eq:iterobj1} and
\eqref{eq:iterobj2} converge to points with destroyed sparsity in the
original eigenvectors. Adding $\ell_1$ penalty to
\eqref{eq:obj2}~\cite{Lu2017} is proposed to achieve the sparsity as much
as possible in DFT problems, which is not likely to be applicable to FCI
problems.

Another way of explicitly getting the eigenpairs rather than a point in
the eigenspace is to solve the single-column version of \eqref{eq:obj1} or
\eqref{eq:obj2} recursively. For example, first, we solve the single
column version of either \eqref{eq:obj1} or \eqref{eq:obj2} for $A_1 = A$
and obtain the smallest eigenpair $\lambda_1$ and $u_1$. Then we apply the
method to $A_2 = A_1 - \lambda_1 u_1 u_1^\top$ and obtain $\lambda_2$ and
$u_2$. At $k$-th time, the method is applied to $A_k = A_{k-1} -
\lambda_{k-1} u_{k-1} u_{k-1}^\top = A - \sum_{i = 1}^{k-1} \lambda_i u_i
u_i^\top$ and $\lambda_k$ and $u_k$ are computed. Such a recursive
procedure has two drawbacks. First, single column operations are composed
of BLAS1-level and  BLAS2-level operations, which are not as efficient as
BLAS3-level operations in modern computer architecture. The second
drawback is the lack of efficient representation of the transformed matrix
$A_k$. The sparsity in $A$ plays a crucial role in designing algorithms
for FCI problems. While, $A_k$ is not as sparse as $A$ in almost all
cases.

Although the aforementioned recursive procedure is not ideal for our
problems, it inspires \TOMOne{} and \TOMTwo{}. We will first motivate and
derive \TOMOne{}. Then \TOMTwo{} can be derived in an analogy way.

In the above recursive procedure, the single column version of
\eqref{eq:iterobj1} is applied to $A_k = A - \sum_{i = 1}^{k-1} \lambda_i
u_i u_i^\top$. Notice that if the column-by-column procedure is applied,
the convergent point of $x_i$ is $\pm \sqrt{-\lambda_i}u_i$. Hence, $A_k$
can be viewed as the summation of $A$ with the outer product of convergent
vector of $x_1$, $x_2$, \dots, $x_{k-1}$. If we assume all columns update
together, and the single column version of \eqref{eq:iterobj1} is applied
to a closed approximation of $A_k$, \ie, $A_k \approx \tA_k = A + \sum_{i
= 1}^{k-1} x_i x_i^\top$, then we obtain the following iterative schemes,
\begin{equation}
    \begin{split}
        x_1^{(t+1)} = & x_1^{(t)} - \alpha \left( A x_1^{(t)} + x_1^{(t)}
        \left( x_1^{(t)} \right)^\top x_1^{(t)} \right), \\
        x_2^{(t+1)} = & x_2^{(t)} - \alpha \left( A x_2^{(t)} +
        x_1^{(t)} \left( x_1^{(t)} \right)^\top x_2^{(t)} + x_2^{(t)}
        \left( x_2^{(t)} \right)^\top x_2^{(t)} \right), \\
        \cdots & \\
        x_k^{(t+1)} = & x_k^{(t)} - \alpha \left( A x_k^{(t)} + \sum_{i
        = 1}^{k} x_i^{(t)}
        \left( x_i^{(t)} \right)^\top x_k^{(t)} \right), \\
        \cdots. & \\
    \end{split}
\end{equation}
Using matrix notations, the above iterative schemes admit the following
representation,
\begin{equation} \label{eq:triofm-obj1}
    X^{(t+1)} = X^{(t)} - \alpha \left( AX^{(t)} + X^{(t)} \triu{ \left(
    X^{(t)} \right)^\top X^{(t)}} \right),
\end{equation}
where $\triu{\cdot}$ denote the upper triangular part of a given matrix.
The key difference between \eqref{eq:iterobj1} and \eqref{eq:triofm-obj1}
is that the gradient is modified as,
\begin{equation} \label{eq:g1}
    g_1(X) = AX + X \triu{X^\top X}.
\end{equation}

Unfortunately, $g_1$ in~\eqref{eq:g1} is not a gradient of any energy
function. Hence, instead of analyzing the stationary points of the
energy function, we analyze the fixed points of \eqref{eq:triofm-obj1}
in Theorem~\ref{thm:stationarypt-obj1}.

\begin{theorem} \label{thm:stationarypt-obj1}
    All {\bfseries fixed points} of \eqref{eq:triofm-obj1} are of form $X =
    U_q \sqrt{-\Lambda_q} P S$, where $\sqrt{\cdot}$ is applied entry-wise,
    $P \in \bbR^{q \times p}$ is the first $p$ columns of an arbitrary
    $q$-by-$q$ permutation matrix, and $S \in \bbR^{p \times p}$ is a
    diagonal matrix with diagonal entries being $0$ or $\pm 1$. Within these
    points all {\bfseries stable fixed points} are of form $X = U_p
    \sqrt{-\Lambda_p} D$, where $D \in \bbR^{p \times p}$ is a diagonal
    matrix with diagonal entries being $\pm 1$. Others are {\bfseries
    unstable fixed points}.
\end{theorem}

\begin{proof}

All fixed points of \eqref{eq:triofm-obj1} satisfy $g_1(X) = 0$ with
$g_1(X)$ being a $n$-by-$p$ matrix. We prove the theorem by induction.
Here we introduce notations in addition to that in
Table~\ref{tab:notations}: $P_i \in \bbR^{q \times i}$ is the first $i$
columns of an arbitrary $q$-by-$q$ permutation matrix, $S_i \in \bbR^{i
\times i}$ is a diagonal matrix with diagonal entries being $0$ or $\pm
1$, and $D_i \in \bbR^{i\times i}$ is a diagonal matrix with diagonal
entries being $\pm 1$.

Consider the first column of $g_1(X) = 0$,
\begin{equation} \label{eq:equalityx1}
    Ax_1 + x_1 x_1^\top x_1 = 0,
\end{equation}
where $x_1^\top x_1$ is a non-negative scalar. When $x_1 = 0$,
\eqref{eq:equalityx1} naturally holds. When $x_1 \neq 0$, $x_1$ must be a
scalar multiple of an eigenvector of $A$ and $x_1^\top x_1$ is the
negative of the corresponding eigenvalue, which must be negative. Hence
$X_1 = x_1$ is of the form, $X_1 = U_q \sqrt{-\Lambda_q} P_1 S_1$.

Now assume the first $i$ columns of $X$ obeys $X_i = U_q
\sqrt{-\Lambda_q} P_i S_i$. Then the $(i+1)$-th column
of $g_1(X) = 0$ obeys
\begin{equation} \label{eq:equalityxi}
    0 = Ax_{i+1} + X_i X_i^\top x_{i+1} + x_{i+1} x_{i+1}^\top x_{i+1}
    = \tA x_{i+1} + x_{i+1} x_{i+1}^\top x_{i+1},
\end{equation}
where $\tA = A + X_i X_i^\top = A + U_q \sqrt{-\Lambda_q} P_i S_i^2
P_i^\top \sqrt{-\Lambda_q} U_q^\top$. $\tA$ is the original matrix $A$
zeroing out a few eigenvalues corresponding to the selected columns in
$P_i$ with $\pm 1$ in $S_i$. Applying the similar analysis as in the case
of \eqref{eq:equalityx1} to \eqref{eq:equalityxi}, we conclude that
$X_{i+1}$ is of the form, $X_{i+1} = U_q \sqrt{-\Lambda_q} P_{i+1}
S_{i+1}$.

Since $q \geq p$, we have a sufficient number of negative eigenpairs to be
added to $X$. The induction can be processed until $i = p$, and we obtain
the expression for all fixed points as in the theorem.

The stabilities of fixed points are determined by the spectrum of their
Jacobian matrices of $g_1$, \ie, $\Diff g_1(X)$. Since both $g_1(X)$ and
$X$ are matrices, the Jacobian is a 4-way tensor, which is unfolded as a
matrix here. In order to avoid over complicated index in subscripts, we
denote the matrix $g_1(X)$ as $G$. Notation $G_{ij}$ and $x_{ij}$ denote
the $(j,i)$-th element of $G$ and $X$ respectively. Then the Jacobian
matrix is written as a $p$-by-$p$ block matrix,
\begin{equation}
    \Diff g_1(X) = \Diff G = 
    \begin{pmatrix}
    J_{11} & \cdots & J_{1p}\\ 
    \vdots & \ddots & \vdots\\ 
    J_{p1} & \cdots & J_{pp}
    \end{pmatrix},
\end{equation}
with block $J_{ij}$ being,
\begin{equation}
    J_{ij} = 
    \begin{pmatrix}
        \parfrac{G_{i1}}{x_{j1}} & \cdots & \parfrac{G_{in}}{x_{j1}}
        \\
        \vdots & \ddots  & \vdots \\
        \parfrac{G_{i1}}{x_{jn}} & \cdots & \parfrac{G_{in}}{x_{jn}}
        \\
    \end{pmatrix}.
\end{equation}
Notice that the $i$-th column of $G$, $G_i = A x_i + X_i X_i^\top x_i$, is
independent of $x_{i+1}, \dots x_p$, which means $J_{ij} = 0$ for $i < j$.
$\Diff g_1(X)$ is a block upper triangular matrix. The spectrum of $\Diff
g_1(X)$ is determined by the spectrum of $J_{ii}$ for $i = 1, 2, \dots,
p$. Through a multivariable calculus, we obtain the explicit expression
for $J_{ii}$,
\begin{equation} \label{eq:Jacobiansubmat}
    J_{ii} = A + X_i X_i^\top + x_i^\top x_i I + x_i x_i^\top.
\end{equation}

We first show the stability of the fixed points of form $X = U_p
\sqrt{-\Lambda_p} D$. Substituting these points into
\eqref{eq:Jacobiansubmat}, we have,
\begin{equation}
    J_{ii} = A - U_i \Lambda_i U_i^\top - \lambda_i I - u_i
    \lambda_i u_i^\top.
\end{equation}
Since $\lambda_i$ is negative and strictly smaller than all eigenvalues of
$A - U_i \Lambda_i U_i^\top$, $J_{ii}$ is strictly positive definite for
all $i = 1, 2, \dots, p$. Therefore we have all eigenvalues of $\Diff
g_1(U_p \sqrt{-\Lambda_p} D)$ are strictly positive and $X = U_p
\sqrt{-\Lambda_p} D$ are stable fixed points.

Next, we show the instability of the rest fixed points.  If $X$ is a fixed
points but not of the form $U_p \sqrt{-\Lambda_p} D$, then there exist
indices $s$ such that $x_s^\top u_s = 0$. Denote $s$ as the first such
index. Substituting this point into $J_{ss}$ and computing the bilinear
form of $J_{ss}$ with respect to $u_s$, we have,
\begin{equation}
    u_s^\top J_{ss} u_s = \lambda_s - x_s^\top x_s < 0,
\end{equation}
where the inequality comes from the fact that $x_s$ is zero or corresponds
to eigenvalues greater than $\lambda_s$. Hence the Jacobian matrix has
negative eigenvalues. Hence these points are unstable fixed points.

\end{proof}

Algorithm~\ref{alg:triofm-obj1} is the pseudocode for
\eqref{eq:triofm-obj1}. The choice of the stepsize is unspecified, which
will be revealed in later sections.

\begin{algorithm}[ht]
    \caption{\TOM{}-\eqref{eq:obj1}/\TOM{}-\eqref{eq:obj2}}
    \begin{algorithmic}
        \State {\bf Input:} a symmetric matrix A,
            an initial point $X^{(0)}$
        \State $t=0$
        \While{not converged}
            \State \begin{empheq}[left={g^{(t)}=\empheqlbrace}]{align}
                & AX^{(t)} + X^{(t)} \triu{\left( X^{(t)}
                \right)^\top X^{(t)}}
                \label{alg:triofm-obj1} \tag{TriOFM-(Obj1)} \\
                & 2AX^{(t)} - AX^{(t)} \triu{
                \left(X^{(t)}\right)^\top X^{(t)}} - X^{(t)} \triu{
                \left(X^{(t)}\right)^\top A X^{(t)}}
                \label{alg:triofm-obj2} \tag{TriOFM-(Obj2)}
            \end{empheq}
            \State Choose a stepsize $\alpha^{(t)}$
            \State $X^{(t+1)} = X^{(t)} - \alpha^{(t)} g^{(t)}$
            \State $t = t + 1$
        \EndWhile
    \end{algorithmic}
\end{algorithm}

There is another way to understand the iterative scheme. The column with a
smaller index is decoupled from columns with larger indices. For example,
the iterative scheme of $x_1$ is independent of all later columns. For the
second column $x_2$, the iterative scheme on $x_2$ is the same as the
second column in the 2-column version of \eqref{eq:obj1}.  Recursively
applying the idea, we also reach Algorithm~\ref{alg:triofm-obj1}.

Similar idea can be applied to solve \eqref{eq:obj2} as well. We notice
that there are two terms in \eqref{eq:iterobj2} coupling columns together,
\ie, $AX^{(t)} \left( X^{(t)} \right)^\top X^{(t)}$ and $X^{(t)} \left(
X^{(t)} \right)^\top AX^{(t)}$. Using the decoupling idea, we can replace
the $\left( X^{(t)} \right)^\top X^{(t)}$ and $\left( X^{(t)} \right)^\top
AX^{(t)}$ by their upper triangular parts and result the following
iterative scheme,
\begin{equation} \label{eq:triofm-obj2}
    X^{(t+1)} = X^{(t)} - \alpha \left( 2AX^{(t)} - AX^{(t)} \triu{
    \left( X^{(t)} \right)^\top X^{(t)}} - X^{(t)} \triu{ \left(
    X^{(t)} \right)^\top AX^{(t)}} \right).
\end{equation}
Comparing to \eqref{eq:iterobj2}, the gradient is modified as,
\begin{equation} \label{eq:g2}
    g_2(X) = 2AX - AX \triu{X^\top X} - X \triu{X^\top AX}.
\end{equation}

The fixed points of \eqref{eq:triofm-obj2} can be analyzed in a similar
way. We summarize the properties in Theorem~\ref{thm:stationarypt-obj2}
and leave the proof in Appendix~\ref{app:stationarypt-obj2}.

\begin{theorem} \label{thm:stationarypt-obj2}
    Let $A$ be a negative definite matrix.  All {\bfseries fixed
    points} of \eqref{eq:triofm-obj2} are of form $X = U P S$ and
    all {\bfseries stable fixed points} are of form $X = U_p D$,
    where $P \in \bbR^{n \times p}$ is the first $p$ columns of an
    arbitrary $n$-by-$n$ permutation matrix, $S \in \bbR^{p \times
    p}$ is a diagonal matrix with diagonal entries being $0$ or $\pm 1$,
    and $D \in \bbR^{p\times p}$ is a diagonal matrix with diagonal
    entries being $\pm 1$.
\end{theorem}

Algorithm~\ref{alg:triofm-obj2} illustrates the pseudocode for
\eqref{eq:triofm-obj2} and the choice of the stepsize is also deferred to
later sections.

We claim a few advantages of Algorithm~\ref{alg:triofm-obj1} and
Algorithm~\ref{alg:triofm-obj2} over other related methods. First, both
algorithms converge to the eigenvectors or their scaled ones without
mixing them. Hence the sparsity of the eigenvectors is preserved. Although
we do not benefit from the sparsity during the iteration in
Algorithm~\ref{alg:triofm-obj1} and Algorithm~\ref{alg:triofm-obj2}
directly, we expect that the coordinate descent methods would benefit from
the sparsity and achieve fast convergence and small memory cost for FCI
problems. Second, the orthogonalization step is totally removed, which
makes the algorithm friendly to parallel computing. Third, all cubic
scaling operations can be processed through BLAS3-level routines.
Algorithms, therefore, benefit from the memory hierarchy of modern
computer architecture.

Although we only propose Algorithm~\ref{alg:triofm-obj1} and
Algorithm~\ref{alg:triofm-obj2} and analyze their convergence in this
paper, the idea of \TOM{} can be applied to a wide range of algorithms to
remove the redundancy introduced by the rotation invariance. The key point
here is decoupling each column from later columns while ensuring that the
iterative scheme for a column remains the same as solving the multicolumn
version of the objective function. The question of where and how \TOM{}
can be applied is open.

\section{Convergence Analysis}
\label{sec:localconv}

In this section, we focus on the local convergence of the proposed
\TOMOne{}. A similar result holds for \TOMTwo{} as well. We denote the set
of stable fixed points as $\calX^*$ and a stable fixed point as $X^* \in
\calX^*$. Further, $x^*_i$ denotes the $i$-th column of $X^*$ and $X_i^*$
denotes the first $i$ columns of $X^*$. The conclusion of the local
convergence to $X^*$ is given in Theorem~\ref{thm:local-conv-obj1},
whereas Lemma~\ref{lem:one-vec-linear} and
Lemma~\ref{lem:multi-vec-linear} provide per-iteration bound on the
residual of the first column and later columns, respectively. Finally, the
rate of local convergence is given in Corollary~\ref{cor:local-conv-obj1}.

\begin{lemma} \label{lem:one-vec-linear}
    Assume the stepsize $\alpha$ satisfies $\alpha < \min_{1 \leq j \leq
    p} \{\frac{1}{4\rho},\frac{1}{\lambda_{j+1}-\lambda_j}\}$. Let
    $\veps_1^{(t)}$ be the error of the first column after the $t$-th
    iteration, $\veps_1^{(t)} = x_1^{(t)} - x_1^*$. If
    $\norm{\veps_1^{(t)}} \leq \frac{\lambda_2 - \lambda_1} {8
    \sqrt{-\lambda_1}}$, then $\norm{\veps_1^{(t+1)}} \leq \left( 1 -
    \alpha \frac{\lambda_2-\lambda_1}{2} \right) \norm{\veps_1^{(t)}}$.
\end{lemma}

\begin{proof}

Without loss of generality, we assume that $A$ is a diagonal matrix. For
simplicity, we drop the iteration index superscript and use $x_1 =
x_1^{(t)}$, $\veps_1 = \veps_1^{(t)}$, $\tx_1 = x_1^{(t+1)}$ and $\teps_1
= \veps_1^{(t+1)}$ instead.  Further we denote the first column of $X^*$
as $v_1 = x_1^*$. From Theorem~\ref{thm:stationarypt-obj1}, we have
$v_1^\top v_1 = -\lambda_1$ and $v_1 v_1^\top = -\lambda_1 u_1 u_1^\top =
-\lambda_1 e_1 e_1^\top$.

Based on the iterative scheme on the first column, \ie, $\tx_1 = x_1 -
\alpha Ax_1 - \alpha x_1^\top x_1 x_1$, we have,
\begin{equation}
    \begin{split}
        \teps_1 = \tx_1 - v_1 = &
        \veps_1 - \alpha A (v_1 + \veps_1) - \alpha (v_1 +
        \veps_1)^\top (v_1 + \veps_1) (v_1 + \veps_1) \\
        = &
        \left( (1 + \alpha \lambda_1) I - \alpha A - 2 \alpha
        v_1v_1^\top \right) \veps_1
        - \alpha v_1 \norm{\veps_1}^2 - 2 \alpha v_1^\top \veps_1
        \veps_1 - \alpha \norm{\veps_1}^2 \veps_1.
    \end{split}
\end{equation}
The assumption on $\alpha$ implies that $1 + \alpha \lambda_1 \pm \alpha
\lambda_i > 0$ holds for all $i$. Hence, the 2-norm of the diagonal matrix
$(1+\alpha \lambda_1) I - \alpha A - 2 \alpha v_1 v_1^\top $ admits
\begin{equation}
    \norm{(1+\alpha \lambda_1) I - \alpha A - 2 \alpha v_1
    v_1^\top } = 1 + \alpha \lambda_1 - \alpha \lambda_2.
\end{equation}
The norm of $\teps_1$ is bounded as,
\begin{equation}
        \norm{\teps_1} \leq ( 1 + \alpha \lambda_1 - \alpha
        \lambda_2 ) \norm{\veps_1} + 3 \alpha \sqrt{-\lambda_1}
        \norm{\veps_1}^2 + \alpha \norm{\veps_1}^3
        \leq \left(1-\alpha\frac{\lambda_2-\lambda_1}{2}\right)
        \norm{\veps_1},
\end{equation}
where the second inequality adopts the fact $\norm{\veps_1} \leq
\frac{\lambda_2-\lambda_1}{8\sqrt{-\lambda_1}}$.

\end{proof}

In Lemma~\ref{lem:one-vec-linear}, we prove that the error of $x_1$
converges linearly in a neighborhood of the stable fixed point. Now we
move on to the multicolumn case. For the $i$-th column $x_i$, we have the
following lemma.

\begin{lemma} \label{lem:multi-vec-linear}
    Assume the stepsize $\alpha$ satisfies $\alpha < \min_{1 \leq j \leq
    p} \{\frac{1}{4\rho},\frac{1}{\lambda_{j+1}-\lambda_j}\}$. Let
    $\veps_i^{(t)}$ be the error of the $i$-th column after the $t$-th
    iteration, $\veps_i^{(t)} = x_i^{(t)} - x_i^*$.  If
    $\norm{\veps^{(t)}_j}\leq
    \frac{\lambda_{j+1}-\lambda_j}{8\sqrt{-\lambda_j}}$ for all $j\leq i$,
    then we have $\norm{\veps^{(t+1)}_i}\leq \left(1-\alpha
    \frac{\lambda_{i+1}-\lambda_i}{2}
    \right)\norm{\veps^{(t)}_i}+\alpha\sum_{j=1}^{i-1}
    \frac{2\norm{A}^2}{\sqrt{\lambda_j \lambda_i}}\norm{\veps^{(t)}_j}$.
\end{lemma}

\begin{proof}

Similarly, we drop the superscript in the proof. We denote the $i$-th
column of $X^*$ as $v_i = x_i^*$. From
Theorem~\ref{thm:stationarypt-obj1}, we have $v_i^\top v_i = -\lambda_i$
and $v_i v_i^\top = -\lambda_i u_i u_i^\top $ for $i = 1, \dots, p$.

Based on the iterative scheme,
$\tx_i = x_i - \alpha A x_i - \alpha \sum_{j=1}^{i}x_j x_j^\top
x_i$, there is
\begin{equation}\label{eq:iter-loc-multi}
    \begin{split}
        \teps_i = \tx_i - v_i = & \veps_i - \alpha A (v_i + \veps_i)
        - \alpha \sum_{j=1}^i (v_j v_j^\top  + \veps_j \veps_j^\top
        + v_j \veps_j^\top + \veps_j v_j^\top )(v_i + \veps_i) \\
        = & ((1 + \alpha \lambda_i)I - \alpha A - \alpha v_i
        v_i^\top - \alpha \sum_{j=1}^i v_j v_j^\top ) \veps_i -
        \alpha \sum_{j=1}^{i-1} v_j v_i^\top  \veps_j \\
        & - \alpha \sum_{j=1}^i \veps_j \veps_j^\top v_i - \alpha
        \sum_{j=1}^i (v_j \veps_j^\top + \veps_j v_j^\top ) \veps_i
        - \alpha \sum_{j=1}^i \veps_j \veps_j^\top \veps_i.
    \end{split}
\end{equation}
The norm of the prefactor of $\veps_i$ can be bounded as,
\begin{equation}
    \norm{\left(1+\alpha \lambda_i\right)I-\alpha A-\alpha v_i
    v_i^\top -\alpha \sum_{j=1}^i v_j v_j^\top } \leq 1+\alpha
    \lambda_i-\alpha \lambda_{i+1}.
\end{equation}

The norm of \eqref{eq:iter-loc-multi} is bounded as,
\begin{equation}
    \begin{split}
        \norm{\teps_i} \leq & \left(1 + \alpha \lambda_i -
        \alpha \lambda_{i+1}\right) \norm{\veps_i} + \alpha
        \sum_{j=1}^{i-1} \sqrt{\lambda_i \lambda_j} \norm{\veps_j}
        + \alpha \sqrt{-\lambda_i} \sum_{j=1}^i \norm{\veps_j}^2 \\
        & + 2 \alpha \sum_{j=1}^i \sqrt{-\lambda_j} \norm{\veps_j}
        \norm{\veps_i} + \alpha \sum_{j=1}^i \norm{\veps_j}^2
        \norm{\veps_i} \\
        = & \left(1 - \alpha \lambda_{i+1} + \alpha \lambda_i
        \right) \norm{\veps_i} + 3 \alpha \sqrt{-\lambda_i}
        \norm{\veps_i}^2 +\alpha \norm{\veps_i}^3 \\
        & +\alpha \sum_{j=1}^{i-1} \left[\sqrt{\lambda_i \lambda_j}
        \norm{\veps_j} + \sqrt{-\lambda_i} \norm{\veps_j}^2 +
        2 \sqrt{-\lambda_j} \norm{\veps_i} \norm{\veps_i} +
        \norm{\veps_j}^2 \norm{\veps_i} \right].
    \end{split}
\end{equation}
Denote $\Delta_j = \lambda_{j+1}-\lambda_j$ as the $j$-th eigengap. Using
the assumption $\norm{\veps_j}\leq \frac{\Delta_j}{8\sqrt{-\lambda_j}}$
for all $1 \leq j \leq i$, we have
\begin{equation}
    \norm{\teps_i} \leq \left(1 - \alpha \frac{\Delta_i}{2}\right)
    \norm{\veps_i} + \alpha \sum_{j=1}^{i-1} \left(
    \sqrt{\lambda_i \lambda_j} + \sqrt{-\lambda_i}
    \frac{\Delta_j}{8 \sqrt{-\lambda_j}} + \sqrt{\lambda_j}
    \frac{\Delta_i}{4 \sqrt{-\lambda_i}} + \frac{\Delta_i
    \Delta_j}{64 \sqrt{\lambda_i \lambda_j}}\right) \norm{\veps_j},
\end{equation}
where the first term $\left(1-\alpha \frac{\Delta_i}{2}\right)
\norm{\veps_i}$ is bounded in the same way as that in
Lemma~\ref{lem:one-vec-linear}. The second term can further be controlled
recursively,
\begin{equation}
    \begin{split}
        \norm{\teps_i} \leq & \left(1 - \alpha
        \frac{\Delta_i}{2}\right) \norm{\veps_i} + \alpha
        \sum_{j=1}^{i-1} \frac{64\lambda_i \lambda_j - 8 \lambda_i
        \Delta_j - 16 \lambda_j \Delta_i + \Delta_i \Delta_j}{64
        \sqrt{\lambda_i \lambda_j}} \norm{\veps_j} \\
        \leq & \left(1 - \alpha \frac{\Delta_i}{2}\right)
        \norm{\veps_i} + \alpha \sum_{j=1}^{i-1}
        \frac{116\norm{A}^2}{64\sqrt{\lambda_i \lambda_j}}
        \norm{\veps_j} \\
        \leq & \left(1 - \alpha \frac{\Delta_i}{2}\right)
        \norm{\veps_i} + \alpha \sum_{j=1}^{i-1}
        \frac{2\norm{A}^2}{\sqrt{\lambda_i \lambda_j}}
        \norm{\veps_j}.
    \end{split}
\end{equation}
Here $\norm{A}$ is adopted to simplify the final bound since all
$\lambda$s are controlled by $\norm{A}$.

\end{proof}

Lemma~\ref{lem:one-vec-linear} is a linear convergence result directly
whereas Lemma~\ref{lem:multi-vec-linear} is slightly different from the
standard linear convergence result. In Theorem~\ref{thm:local-conv-obj1}
we investigate the extra term $2\alpha \frac{\norm{A}^2}{\sqrt{\lambda_j
\lambda_i}} \norm{\veps_j}$ and find out that the overall local
convergence is linear.

\begin{theorem} \label{thm:local-conv-obj1}
    Assume the stepsize $\alpha$ satisfies $\alpha < \min_{j\in \left[1,p
    \right]}\{\frac{1}{4\rho},\frac{1}{\lambda_{j+1}-\lambda_j}\}$. Let
    $\veps_i^{(t)}$ be the error of the $i$-th column after the $t$-th
    iteration, $\veps_i^{(t)} = x_i^{(t)} - x_i^*$. If
    $\norm{\veps^{(0)}_j} \leq \frac{\lambda_{j+1} - \lambda_j}{8
    \sqrt{-\lambda_j}}$ for all $j \leq i$, then for any $i = 1, \dots, p$
    there exists a polynomial $C_i(t)$ of degree $i-1$ such that
    \begin{equation} \label{eq:local-rate}
        \norm{\veps_i^{(t)}}\leq C_i(t) r_i^{t}.
    \end{equation}
    where $r_i = 1 - \frac{\alpha}{2} \min_{j \in [1,i]} \{\lambda_{j+1} -
    \lambda_j\}$.
\end{theorem}

\begin{proof}

The theorem is proved by induction. First, Lemma~\ref{lem:one-vec-linear}
shows that $\norm{\veps^{(t)}_1}$ satisfies \eqref{eq:local-rate} for $C_1
= \norm{\veps^{\left(0 \right)}_1}$. Given $i \leq p$, we assume that the
theorem holds for all $j < i$. We further assume that all polynomials in
the theorem are non-decreasing. Denoting $a_j = \alpha \frac{2
\norm{A}^2}{\sqrt{\lambda_j \lambda_i}}$, the inequality in
Lemma~\ref{lem:multi-vec-linear} can be further bounded as,
\begin{equation}
    \norm{\veps_i^{(t)}} \leq r_i \norm{\veps_i^{(t-1)}}
    + \sum_{j=1}^{i-1} a_j \norm{\veps_j^{(t-1)}} \leq r_i
    \norm{\veps_i^{(t-1)}} + C_{max}(t) r_{i-1}^{t-1},
\end{equation}
where $C_{max}(t) = \sum_{j=1}^{i-1} a_j C_j(t)$ and the relationship
$r_1\leq \cdots \leq r_{i-1}$ is used so that all $r_j$s are bounded by
$r_{i-1}$. Notice that for each $j$, $a_j$ is positive and $C_j(t)$ is a
non-decreasing polynomial of degree $j-1$. $C_{max}(t)$ is then a
non-decreasing polynomial of degree $i-2$.

Since the inequality above holds for all $t \geq 1$, we apply it
repeatedly and obtain,
\begin{equation}
    \norm{\veps_i^{(t)}} \leq r_i^{t} \norm{\veps_i^{(0)}} +
    \sum_{k = 0}^{t-1} r_i^{t-1-k} C_{max}(k) r_{i-1}^{k} \leq
    \left(\norm{\veps_i^{(0)}} + \frac{t}{r_i} C_{max}(t) \right)
    r_i^{t} = C_i(t) r^t_i,
\end{equation}
where $C_i(t) = \norm{\veps_i^{(0)}} + \frac{t}{r_i}C_{max}(t)$ is a
non-decreasing polynomial of degree $i-1$. Hence the theorem is proved.

\end{proof}

Theorem~\ref{thm:stationarypt-obj1} states that there are a set of stable
fixed points of \ref{alg:triofm-obj1}. Next
Corollary~\ref{cor:local-conv-obj1} shows that the iterative scheme
\ref{alg:triofm-obj1} locally has linear convergence to the set of stable
fixed points. We define the distance from a point to a set as, $\fnorm{X -
\calX^*} = \min_{X^* \in \calX^*} \fnorm{X - X^*}$.

\begin{corollary} \label{cor:local-conv-obj1}
    Assume the stepsize $\alpha$ satisfies $\alpha < \min_{j\in \left[1,p
    \right]}\{\frac{1}{4\rho},\frac{1}{\lambda_{j+1}-\lambda_j}\}$. Let
    $\delta^{(t)}$ be the distance from the stable fixed points after the
    $t$-th iteration, $\delta^{(t)} = \fnorm{X^{(t)} - \calX^*}$. If
    $\delta^{(0)} \leq \min_{j \in [1, p]} \frac{\lambda_{j+1} -
    \lambda_j}{8 \sqrt{-\lambda_j}}$, then there exists a polynomial
    $C(t)$ of degree $p-1$ such that $\delta^{(t)} \leq C(t) r^{t}$, where
    $r = 1 - \frac{\alpha}{2} \min_{j \in [1,p]} \{\lambda_{j+1} -
    \lambda_j\}$.
\end{corollary}

\begin{proof}

We first notice that for any two distinct points in $\calX^*$, the
smallest distance in F-norm is $2\sqrt{-\lambda_p}$, which is greater than
twice initial error $\delta^{(0)} \leq \min_{j \in [1, p]}
\frac{\lambda_{j+1} - \lambda_j}{8 \sqrt{-\lambda_j}}$. Hence for any
initial point, it can only be attracted by one stable fixed point. By the
definition of $\delta^{(t)}$ and Theorem~\ref{thm:local-conv-obj1}, we
have,
\begin{equation}
    \delta^{(t)}=\sqrt{\sum_{i=1}^{p}\norm{\veps_i^{(t)}}^2}
    \leq \sqrt{\sum_{i=1}^{p}\left(C_i(t)r_i^t\right)^2}
    \leq \sum_{i=1}^p C_i(t) r_p^t =C(t)r^t
\end{equation}
where $C(t)=\sum_{i=1}^p C_i(t)$ is a polynomial of degree $p-1$ and the
second inequality is due to the non-negativity of $C_i(t)$.

\end{proof}

We shall notice that the estimation $\delta^{(t)} \leq C(t) r^{t}$
satisfies $\lim_{t \rightarrow \infty} \frac{C(t+1) r^{t+1}}{ C(t) r^t} =
r \lim_{t \to \infty} \frac{C(t+1)}{C(t)} = r$, which is the definition of
linear convergence. Hence we claim the iterative scheme
\ref{alg:triofm-obj1} locally converges linearly to stable fixed points. A
similar proof procedure can be applied to show the local linear
convergence for \ref{alg:triofm-obj2}.

\begin{remark}

We noticed that, for general eigensolvers, the convergence rate of each
eigenvector is shift-invariant cause it only relies on the spectrum gap.
However, things are different for the objective functions in this paper,
since the spectrum of \eqref{eq:iterobj1} and \eqref{eq:iterobj2} can be
regarded as the spectrum of $A$ combined with an extra $0$. So, if we
shift the spectrum of $A$ far away from $0$, the algorithms in our paper
would converge slower.

\end{remark}

In addition to local convergence, \ref{alg:triofm-obj1} and
\ref{alg:triofm-obj2} also converge globally. In our companion
paper~\cite{Gao2021}, the global convergence of \ref{alg:triofm-obj1} and
is proved in detail, which is rephrased in Theorem~\ref{thm:global-conv}.
Similar global and local convergence results of
\ref{alg:triofm-obj2}~\cite{Liu2021} have been proved in \cite{Liu2021},
following the idea in~\cite{Gao2021} and this paper. Both results rely on
the stable manifold theorem for discrete dynamical systems.

\begin{theorem} \label{thm:global-conv}
    If the initial point $X^{(0)} = \begin{pmatrix}x_1^{(0)} & x_2^{(0)} &
    \cdots & x_p^{(0)}\end{pmatrix}$ satisfies $\norm{x_i^{(0)}} \leq R_i$
    for all $1\leq i\leq p$, where $R_i=2^{i-1}\sqrt{3\rho}$ and the
    stepsize satisfies $\alpha \leq \frac{1}{10R_p^2}$, then the fixed
    stepsize version of Algorithm~\ref{alg:triofm-obj1} converges to
    $\calX^*$ for all initial points besides a set of measure zero.
\end{theorem}

Comparing the local and global convergence, as in
Corollary~\ref{cor:local-conv-obj1} and Theorem~\ref{thm:global-conv}, the
restrictions on stepsizes are different, \ie, the stepsize of global
convergence is much smaller. Such a small stepsize is needed in global
convergence to overcome the unbounded Lipschitz constant of the underlying
objective function but would be more flexible in practice, especially when
the stepsize is chosen in a sophisticated way, as stated in
Section~\ref{subsec:stepsize}.

\section{Implementation Details}
\label{sec:implementation}

In previous sections, we introduce \TOM{} algorithms based on the gradient
descent method with a constant stepsize and prove their convergence
properties. \TOM{} can be regarded as a modified gradient descent method.
In this section, we explore traditional accelerating techniques for
gradient methods and adapt them to \TOM{}. Such techniques include
momentum acceleration, stepsize choices, and column locking.

\subsection{Momentum Acceleration}

Momentum is a widely-used technique to accelerate gradient descent
methods. In traditional gradient descent methods, with the help of
momentum, the oscillatory trajectory could be smoothed, and the convergence
rate depends on the square root of the condition number rather than the
condition number.

Momentum method, instead of moving along the gradient direction directly,
moves along with an accumulation of gradient directions with a discounting
parameter $\beta \in (0, 1]$, \ie,
\begin{equation}\label{eq:momentum}
    V^{(t)} = \beta g \left( X^{(t)} \right) + (1 - \beta ) V^{(t-1)},
\end{equation}
where $V^{(t)}$ denotes the accumulated direction and $g$ is the gradient.
Then the iteration moves along $V^{(t)}$ with a stepsize
$\alpha$, \ie, $X^{(t+1)} = X^{(t)} - \alpha V^{(t)}$. Since $V$ is a
linear combination of gradient directions, an explicit way to generalize
it to the triangularized method is to replace the gradient $g$ by our
triangularized direction function either $g_1$ or $g_2$. Then we obtain
the momentum accelerated algorithms for \TOMOne{} and \TOMTwo{}.

Importantly, such a modification will not change the dependency among
columns of $X$. With this momentum enabled, the first $i$ columns remain
the same as the algorithm applied on $X = \begin{pmatrix} x_1 & x_2 &
\cdots & x_i \end{pmatrix}$. Any column of $X$ still depends only on
columns on its left throughout the iterations. However, for momentum
methods, choosing an efficient momentum parameter $\beta$ is an art.

Similarly, we can adopt the idea of conjugate
gradient~(CG)~\cite{Golub2013} to triangularized algorithms as well. CG is
a momentum method with adaptive momentum parameters and hence choosing
$\beta$ is avoided. CG is widely applied to solve both linear and
nonlinear problems. The success of nonlinear CG in solving eigenvalue
problems have already been demonstrated in OMM~\cite{Corsetti2014}. A
typical non-linear CG method is the Polak-Reeves CG
(PR-CG)~\cite{M2AN_1969__3_1_35_0}, which adopts the following steps per
iteration in a single-vector setting:
\begin{equation} \label{eq:cg-single}
    \begin{split}
        \beta^{(t)} = & \frac{\left(g(x^{(t)}) - g(x^{(t-1)}) \right)^\top
        g(x^{(t)})}{g(x^{(t-1)})^\top g(x^{(t-1)})}, \\
        v^{(t)} = & -g(x^{(t)}) + \beta^{(t)} v^{(t-1)}, \\
        x^{(t+1)} = & x^{(t)} + \alpha v^{(t)}. \\
    \end{split}
\end{equation}
In a multi-vector setting, \ie, the iteration variable is a matrix, the
formula for $\beta^{(t)}$ could be extended. However, the multi-vector
version for $\beta^{(t)}$ mixes all columns together and destroys the
column dependency of \TOM{}.


A more favorable choice of $\beta^{(t)}$ for \TOM{} is to use different
$\beta^{(t)}$s for different columns, which is called the columnwise CG
throughout this paper. The parameter for the $i$-th column, denoted as
$\beta_i^{(t)}$, is calculated as the single-vector setting with
$x_i^{(t)}$ and applied to update $x_i^{(t)}$. The corresponding algorithm
for \TOM{} is summarized as Algorithm~\ref{alg:triu-columnwisecg}.
In Algorithm~\ref{alg:triu-columnwisecg}, $g_i^{(t)}$ and $v_i^{(t)}$
denote the $i$-th column of $G^{(t)}$ and $V^{(t)}$ respectively.

\begin{algorithm}[ht]
    \caption{Columnwise CG for \TOM{}}
    \label{alg:triu-columnwisecg}
    \begin{algorithmic}
        \State {\bf Input:} symmetric matrix $A$, initial point
        $X^{(0)}$, stepsize $\alpha$
        \State $G^{(0)} = g(X^{(0)})$
        \State $V^{(0)} = - G^{(0)}$
        \State $X^{(1)} = X^{(0)} + \alpha V^{(0)}$
        \State $t = 1$
        \While{not converged}
            \State $G^{(t)} = g(X^{(t)})$
            \For{$i = 1, 2, \dots, p$}
                \State $\beta_i^{(t)} = \frac{(g_i^{(t)}-g_i^{(t-1)})^\top
                g_i^{(t)}}{(g_i^{(t-1)})^\top g_i^{(t-1)}}$
                \State $v_i^{(t)} = - g_i^{(t)} + \beta_i^{(t)} v_{i-1}^{(t)}$
            \EndFor
            \State $X^{(t+1)} = X^{(t)} + \alpha V^{(t)}$
            \State $t = t+1$
        \EndWhile
    \end{algorithmic}
\end{algorithm}


As a remark, there is another way in computing the parameter
$\beta_i^{(t)}$s, \ie, $\beta_i^{(t)}$ is calculated using the
multi-vector version of \eqref{eq:cg-single} with $X_i^{(t)}$. The
dependencies among columns are preserved. However, the calculation must be
conducted carefully to avoid increasing the computational cost.

\subsection{Stepsizes}
\label{subsec:stepsize}

In previous sections, we describe algorithms with a constant stepsize to
simplify the presentation. However, we find that a linesearch strategy
could significantly outperform the constant stepsize. In this section, we
introduce an exact linesearch strategy as the suggested stepsize strategy.

Since both \eqref{eq:obj1} and \eqref{eq:obj2} are quartic polynomials of
$X$, the exact linesearch can be calculated through minimizing quartic
polynomials. Minimizing a quartic polynomial with a positive leading
coefficient is equivalent to solve a cubic polynomial. Taking
\eqref{eq:obj1} as an example, the cubic polynomial is,
\begin{equation} \label{eq:full-linesearch}
    \begin{split}
        \frac{\diff }{\diff \alpha} f_1(X + \alpha V) = & \trace{V^\top
        \grad{f_1(X + \alpha V)}}\\
        = & \alpha^3 \trace{(V^\top V)^2} + 3 \alpha^2 \trace{V^\top V
        X^\top V} \\
        & + \alpha \trace{V^\top A V + (V^\top X)^2 + V^\top X X^\top V
        + V^\top V X^\top X} \\
        & + \trace{V^\top AX} + \trace{V^\top X X^\top X}.
    \end{split}
\end{equation}
Solving the expression above would give possibly one, two, or three real
roots. The best stepsize can be selected among real roots through a basic
analysis~\cite{Li2019c}. Similar calculation and analysis can also be
carried out for \eqref{eq:obj2}. We omit the details here.

However, the stepsize in \eqref{eq:full-linesearch} does not work for
\TOM{}. Consider a simple case for example. If $X$ is in the space spanned
by the smallest eigenpairs but not the stable fixed point, \ie, $X = U_p
\sqrt{-\Lambda_p} Q$ for $Q$ being a non-diagonal unitary matrix, then $X$
is already a global minimum of \eqref{eq:obj1} and the stepsize $\alpha$
is zero from solving \eqref{eq:full-linesearch}. This simple example shows
that the above linesearch strategy is not working properly for \TOM{} and
we need to find a different strategy for the stepsize.

Notice that the exact linesearch solves $\trace{V^\top \grad{f(X+\alpha
V)}} = 0$ for the stepsize $\alpha$. However, \TOM{} adopts $g_1$ or $g_2$
rather than $\grad{f_1}$ or $\grad{f_2}$, which means the iteration is not
consistency with the linesearch \eqref{eq:full-linesearch}. The columnwise
stepsize strategy is as follows. First, consider the stepsize for $x_1$.
We solve two identical equations, $v_1^\top g(x_1 + \alpha v_1) = 0$ and
$v_1^\top \grad{f}(x_1 + \alpha v_1) = 0$, to obtain the stepsize. Now we
consider the stepsize $\alpha_i$ for $x_i$. We can solve either
$\trace{V_i^\top \grad{f(X_i + \alpha_i V_i)}} = 0$ or $\trace{V_i^\top
g(X_i + \alpha_i V_i)} = 0$ for $\alpha_i$.  The former is the same as
\eqref{eq:full-linesearch} with $X$ and $V$ replaced by $X_i$ and $V_i$
respectively. The later can be expressed as again a cubic polynomial of
$\alpha_i$,
\begin{equation} \label{eq:linesearch-poly}
    \begin{split}
        p(\alpha_i) = & \alpha_i^3 \trace{V_i^\top V_i \triu{V_i^\top
        V_i}} \\
        & + \alpha_i^2 \trace{V_i^\top V_i \triu{X_i^\top V_i} +
        V_i^\top V_i \triu{V_i^\top X_i} + V_i^\top X_i \triu{V_i^\top
        V_i}} \\
        & + \alpha_i \trace{V_i^\top A V_i + V_i^\top X_i \triu{V_i^\top
        X_i} + V_i^\top X_i \triu{X_i^\top V_i} + V_i^\top V_i
        \triu{X_i^\top X_i}} \\
        & + \trace{V_i^\top A X_i + V_i^\top X_i \triu{X_i^\top X_i}}.
    \end{split}
\end{equation}
Using either equation, we are able to avoid $\alpha_i = 0$ if $X_i$ stays
in the space spanned by eigenvectors while $X_i$ is not any stable fixed
point. The local convergences for both choices of stepsize can be proved
in a similar way as in Section~\ref{sec:localconv}. Regarding the
computational cost, since all trace terms can be computed in an
accumulative way, the computational cost for getting coefficients in
\eqref{eq:linesearch-poly} and \eqref{eq:full-linesearch} remains the same
for all $i$.

According to our numerical experiments, the columnwise stepsize strategy
based on the linesearch significantly outperforms the fixed stepsize,
while there is not much difference between solving $\trace{V_i^\top g(X_i
+ \alpha V_i)} = 0$ and $\trace{V_i^\top \grad{f(X_i + \alpha V_i)}} = 0$.
Throughout the rest paper, we solve $\trace{V_i^\top g} = 0$ for stepsize.

\subsection{Column Locking}

In Section~\ref{sec:localconv} we notice that each column has its own
convergence rate, and later columns converge slower than earlier ones in
terms of the analysis. A similar conclusion is observed numerically. It
wastes computation resources if all columns are updated throughout
iterations. Hence in addition to the overall stopping criterion of \TOM{}
methods, $\fnorm{g(X^{(t)})} < \epsilon$, we introduce a column locking
technique to allow early stopping for converged columns.

The column locking has been widely adopted in many traditional
eigensolvers. However, in orthogonalization-free
eigensolvers~\cite{Corsetti2014, Li2019c, Wang2019}, the locking technique
is not applicable since all columns are coupled together throughout
iterations. \TOM{}, differently, can adopt the column locking in a
specific ordering. Since the earlier columns in \TOM{} are independent of
later columns, as long as they have converged, we could lock these
columns.

The column locking strategy depends on the error propagation among
columns. Lemma~\ref{lem:multi-vec-linear} hints the error propagation.
However, we find that the error estimation in
Lemma~\ref{lem:multi-vec-linear} is pessimistic. Here we give an intuitive
but helpful discussion on the error propagation, where higher order terms
in the error vector are ignored. Let $\veps_i^{(t)} = \left(
\veps_{i,1}^{(t)}, \cdots, \veps_{i,n}^{(t)} \right)^\top$ be the error of
$x_i^{(t)}$ projected to eigenvectors of $A$, \ie, $\veps_i^{(t)} = U^\top
\left(x_i^{(t)}-x_i^* \right)$. The projected error $\veps_i^{(t)}$ here
is consistent with the notations in Section~\ref{sec:localconv}, where $A$
is assumed to be diagonal. The error in $i$-th column of
\eqref{alg:triofm-obj1}, without higher order terms, admits,
\begin{equation}\label{eq:epsij}
    \veps_i^{(t+1)}=
    \begin{pmatrix}
        (1 + \alpha \lambda_i)\veps_{i,1}^{(t)} - \alpha
        \sqrt{\lambda_1 \lambda_i} \veps_{1,i}^{(t)}\\ 
        \vdots\\ 
        (1 + \alpha \lambda_i)\veps_{i,i-1}^{(t)} - \alpha
        \sqrt{\lambda_{i-1} \lambda_i} \veps_{i-1,i}^{(t)}\\ 
        (1 + 2\alpha \lambda_i) \veps_{i,i}^{(t)}\\ 
        (1 + \alpha (\lambda_{i+1} - \lambda_{i}))
        \veps_{i,i+1}^{(t)}\\ 
        \vdots\\ 
        (1 + \alpha (\lambda_{n} - \lambda_{i}))
        \veps_{i,n}^{(t)}
    \end{pmatrix}.
\end{equation}
The equation \eqref{eq:epsij} implies that the lower triangular part of
$\left(\veps_{1}^{(t)}, \cdots, \veps_{p}^{(t)}\right)$ does not depend on
other error vectors and thus is able to converge to zero as $t$ goes to
infinity even if other columns are locked. For the strict upper triangular
part, we consider the case where columns earlier than $i$ are locked with
fixed errors. Taking the $j$-th row ($j<i$) for example, when
$\veps_{j,i}^{(t)}$ is fixed, $\veps_{i,j}^{(t+1)} = (1 + \alpha
\lambda_i) \veps_{i,j}^{(t)} - \alpha \sqrt{\lambda_j \lambda_i}
\veps_{j,i}^{(t)}$ has the fixed point $\veps_{i,j} =
\sqrt{\frac{\lambda_j}{\lambda_i}}\veps_{j,i}$ as $t$ goes to infinity.
Notice that each entry in the upper triangular part is only influenced by
an error term in the lower triangular part. Through a detailed derivation
by induction, we have an estimation on the norms of error vectors for
\TOMOne{} as $t \to \infty$,
\begin{equation}
    \norm{\veps_1} \sim \epsilon, \quad \norm{\veps_2} \sim
    \sqrt{\frac{\lambda_1}{\lambda_2}}\epsilon, \quad \cdots,
    \quad \norm{\veps_p} \sim \sqrt{\frac{\lambda_1}{\lambda_p}} \epsilon.
\end{equation}
The estimation above shows that there is a uniform upper bound on
$\sqrt{-\lambda_i}\norm{\veps_i}$ for all $1\leq i\leq p$.

Further analysis on $g_1(X^{(t)})$ in the stopping criterion shows that
the norms of columns of $g_1(X^{(t)})$ admit the same scaling as that of
$\norm{\veps_1}, \dots, \norm{\veps_p}$. Hence we could include an
additional term with scaling $\sqrt{-\lambda_i}$ for the $i$-th column. A
good choice is $\norm{Ax_i}^{\frac{1}{3}}$. The recommended locking
criterion for \TOMOne{} is
\begin{equation}
    \norm{g_1(x_i^{(t)})}\norm{Ax_i^{(t)}}^{\frac{1}{3}} < \epsilon.
\end{equation}

An analog estimation can be carried out for \TOMTwo{} as well. The unified
locking criterion for \TOMTwo{} is 
\begin{equation}
    \norm{g_2(x_i^{(t)})}\norm{Ax_i} < \epsilon.
\end{equation}

\section{Numerical Results}
\label{sec:numerical}

In this section, we show the efficiency of \TOM{} applying to three
different groups of matrices, \ie, random matrices with different
eigenvalue distributions, a synthetic matrix from DFT, and a matrix of
Hubbard model under FCI framework.

In Section~\ref{sec:rand_matrix}, we first show that \TOM{} with a
constant stepsize locally has linear convergence rate on random matrices
with different eigenvalue distributions, which agrees with our analysis in
Section~\ref{sec:localconv}. Further, accelerating techniques introduced
in Section~\ref{sec:implementation} are adopted and compared. Then we
apply \TOM{} with these techniques to two matrices from DFT and FCI in
Section~\ref{sec:numres_dft} and Section~\ref{sec:numres_fci}
respectively. In both examples, \TOM{} converges to sparse eigenvectors,
whereas traditional orthogonalization-free methods fail to recover the
sparsity. Regarding the computational cost, \TOM{} is, in general,
comparable to its non-triangularized counterpart.

For a fair comparison reason, we adopt the same stopping criterion for
both \TOM{} and \OFM{}: the relative residual is smaller than a tolerance
$\epsilon$, \ie, $\frac{\fnorm{AXQ-XQ\Lambda_X}}{\fnorm{AXQ}} < \epsilon$,
where $Q$ and the diagonal matrix $\Lambda_X$ come from solving a
generalized eigenvalue problem, $\left(X^\top AX\right)Q=\left(X^\top
X\right)Q\Lambda_X$. Such a stopping criterion is not applicable in
practice. For the illurstration purpose, it is adopted in this section for
a fair comparison. If the column locking is enabled in \TOM{}, the
algorithm could stop if all columns are locked. Two measurements of
accuracies are used. The first one measures the accuracy of eigenvectors,
\begin{equation} \label{eq:evec}
    e_{vec} = \min_{ X^* \in
    \calX^*} \frac{\fnorm{X - X^*}}{\fnorm{X^*}},
\end{equation}
where $\calX^*$ denotes the set of all possible stable fixed points of the
used algorithm. The second measures the accuracy of eigenvalues,
\begin{equation} \label{eq:eval}
    e_{val} = \frac{\abs{\trace{\left(X^\top X\right)^{-1} X^\top A X}
    - \sum_{i=1}^{p} \lambda_i}}{\abs{\sum_{i=1}^{p} \lambda_i}}.
\end{equation}

We also define two measurements for computational costs. Since all of our
codes are implemented in MATLAB, which favors matrix operations over
vector operations, the runtime comparison is not fair. Hence we introduce
\emph{number of iterations} and \emph{number of matrix-vector
multiplications}. Without column locking, the number of matrix-vector
multiplications is simply the number of iterations multiplying the number
of columns in $X^{(t)}$. When column locking is enabled, it is the
summation of the number of unlocked columns throughout iterations.

\subsection{Random Matrices}
\label{sec:rand_matrix}

In this section we apply different \TOM{} algorithms to random matrices
and compare the performance against their \OFM{} counterparts. We generate
random matrices of size $n = 500$. The number of desired eigenpairs is $p
= 5$ and $p=10$ in Section~\ref{sec:numres_localconv} and
Section~\ref{sec:numres_techniques} respectively. Random matrices are
generated of the form
\begin{equation}
    A = U^\top \Lambda U,
\end{equation}
where $U$ is a random orthogonal matrix generated by a QR factorization of
a random matrix with entries sampled from a standard normal distribution
independently. Here $\Lambda$ denotes a diagonal matrix with its elements
$\{\lambda_i\}_{i=1}^n$ generated from three different ways,
\begin{enumerate}
    \item (Uniform) $\lambda_i = \frac{i-1}{500} - 1$ for $1\leq
        i\leq n$;
    \item (Logarithm) $\lambda_i = - \frac{2^{10}}{500} \frac{1}{2^i}$
        for $1\leq i\leq n$;
    \item (U-Shape) $\lambda_1 = -\frac{14}{16}, \lambda_2 =
        -\frac{10}{16}, \lambda_3 = -\frac{8}{16}, \lambda_4 =
        -\frac{7}{16}, \lambda_5 = -\frac{5}{16}, \lambda_i =
        -\frac{1}{16}$ for all $6\leq i\leq n$.
\end{enumerate}
In the U-shape case, the first 5 eigengaps are $\frac{4}{16},
\frac{2}{16}, \frac{1}{16}, \frac{2}{16}, \frac{4}{16}$, which decays
exponentially first and then grows exponentially. We denote these three
random matrices as $A_{uni}$, $A_{log}$, and $A_{ushape}$. The eigengaps
of $A_{uni}$ and $A_{log}$ are two typical cases for many applications.
While, $A_{ushape}$ is constructed to reveal the difference between \TOM{}
and \OFM{}.

\subsubsection{Local Convergence Rate}
\label{sec:numres_localconv}

We first numerically validate the convergence rate proved in
Section~\ref{sec:localconv}. The stepsize is fixed, $\alpha = 0.4$.
Initial state $X^{(0)} \in \bbR^{n\times p}$ is a random matrix with unit
column lengths. Column locking technique is applied, whereas momentum
techniques are disabled.

\begin{figure}[htp]
    \includegraphics[width=0.333\textwidth]{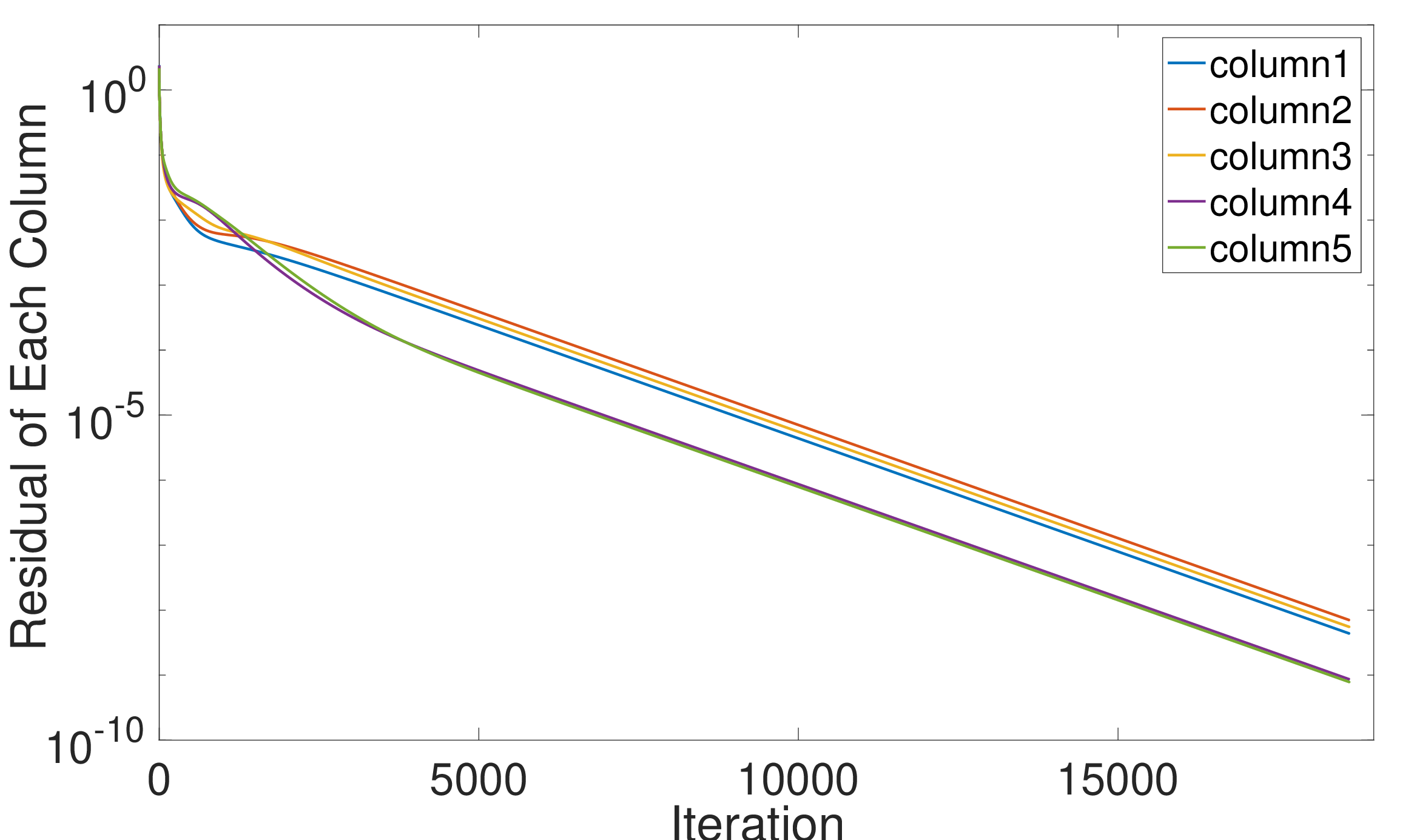}
    \includegraphics[width=0.333\textwidth]{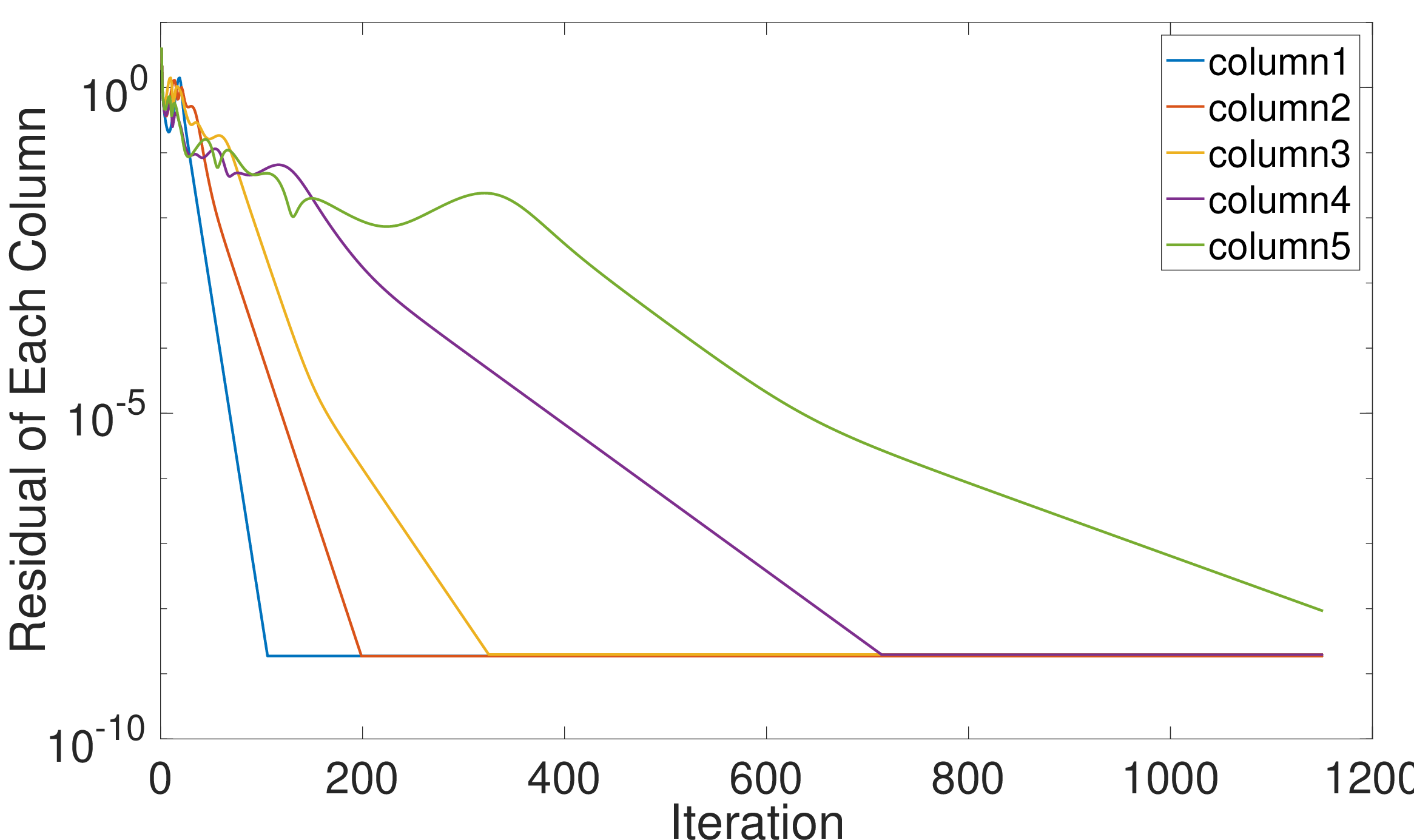}
    \includegraphics[width=0.333\textwidth]{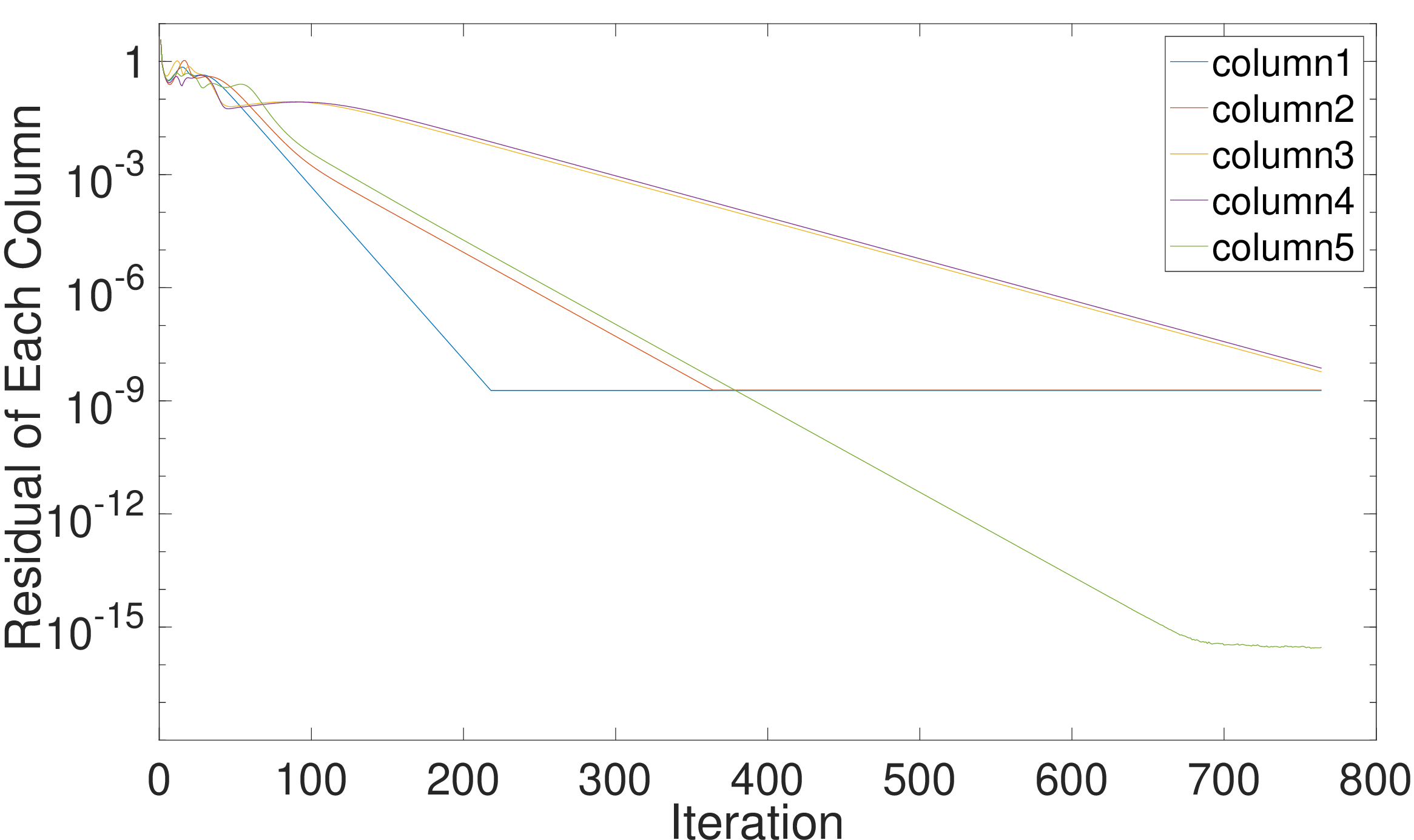}
    \caption{Convergence behavior of \TOMOne{} applying to $A_{uni}$
    (left), $A_{log}$ (middle), and $A_{ushape}$ (right) with fixed
    stepsize $\alpha = 0.4$ and column locking.}
    \label{fig:converge_rate_fixed_stepsize}
\end{figure}

\begin{figure}[htp]
    \includegraphics[width=0.333\textwidth]{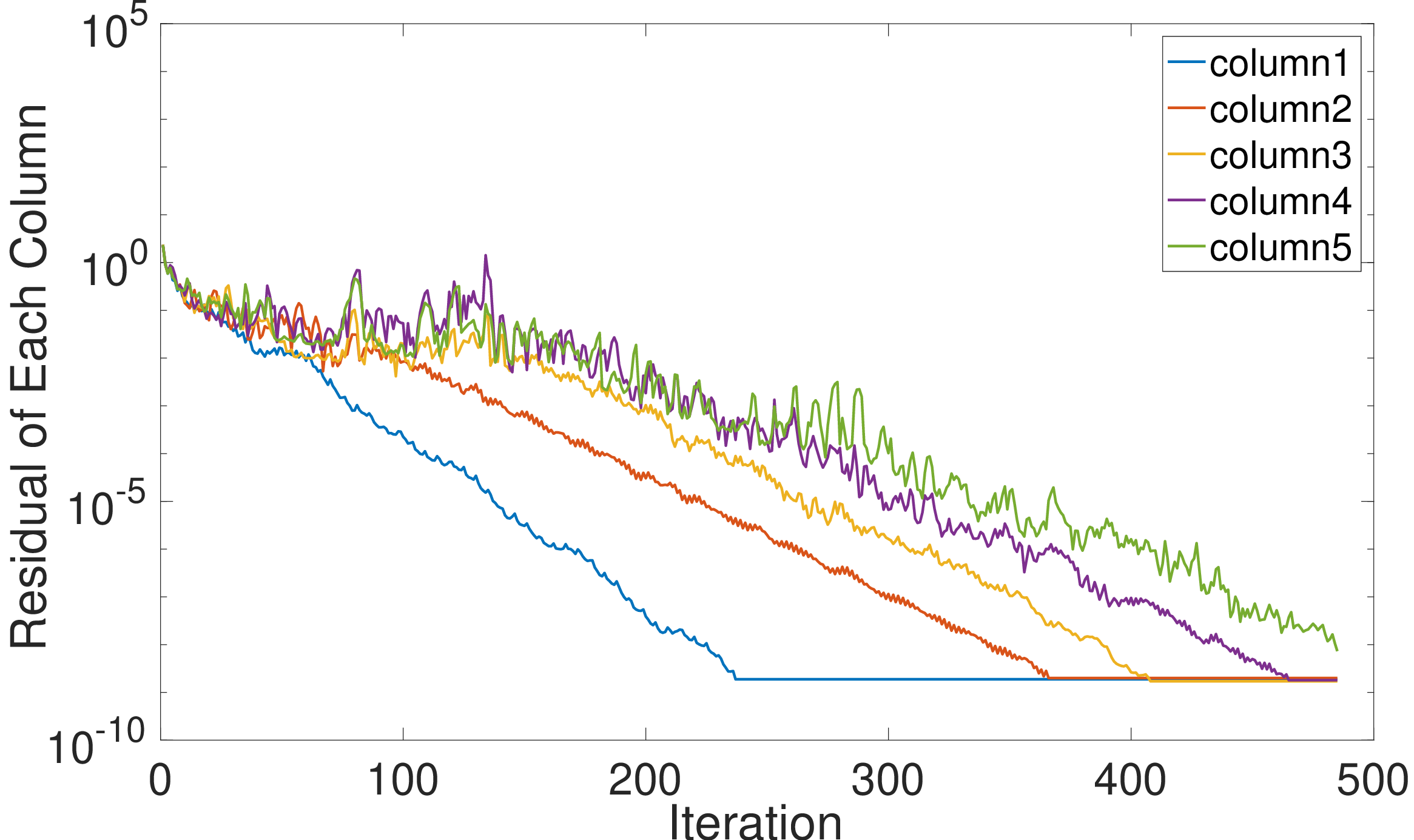}
    \includegraphics[width=0.333\textwidth]{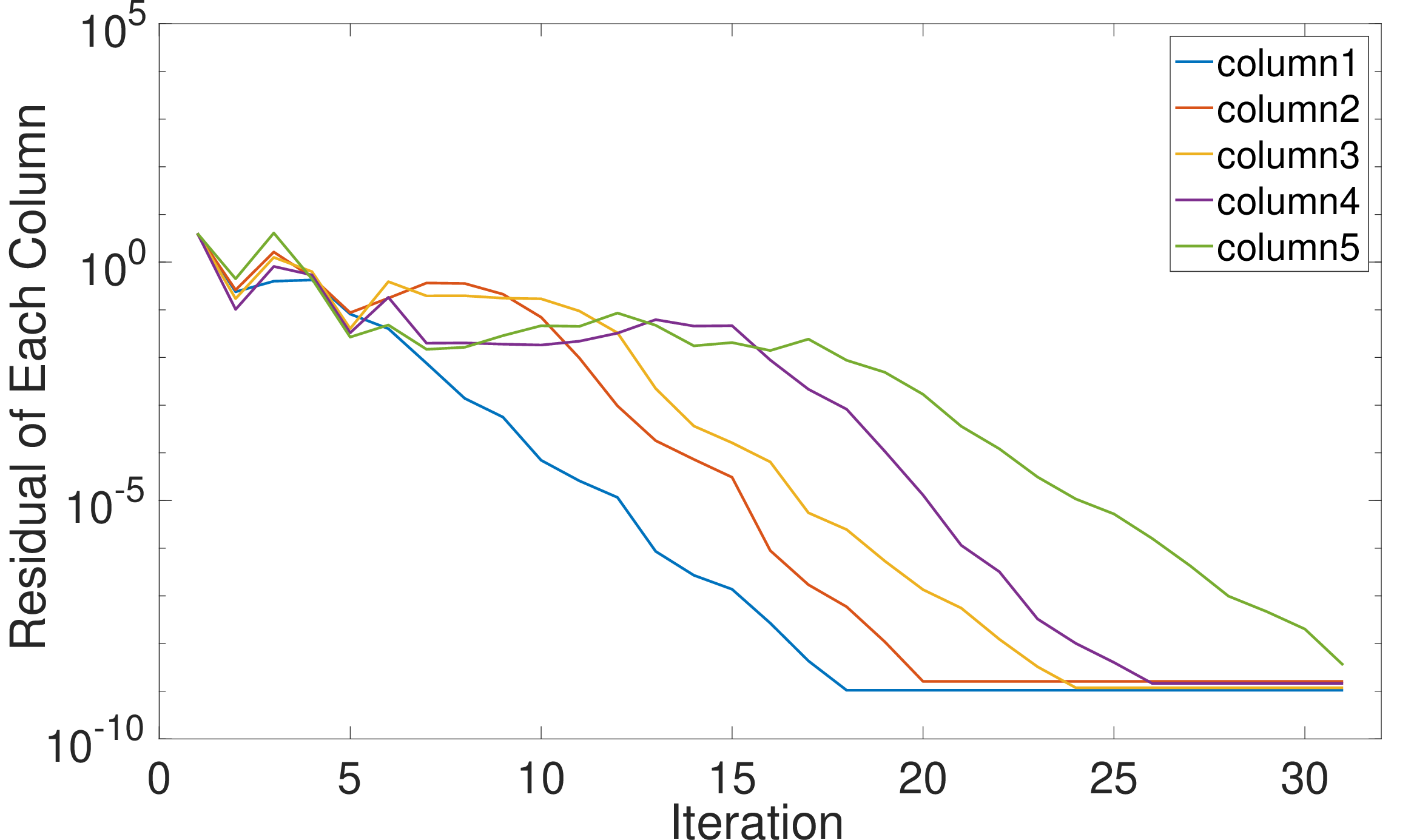}
    \includegraphics[width=0.333\textwidth]{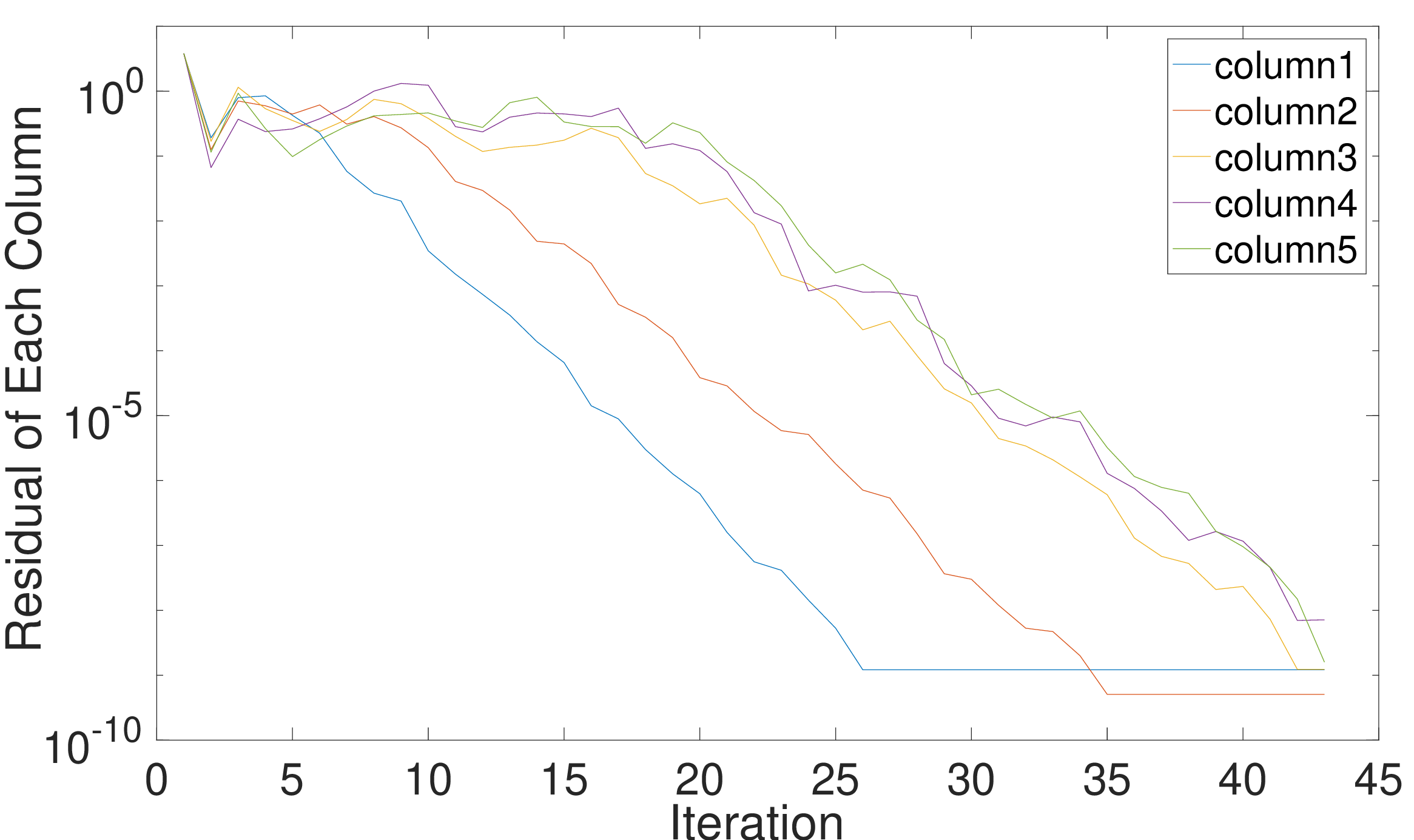}
    \caption{Convergence behavior of \TOMOne{} applying to $A_{uni}$
    (left), $A_{log}$ (middle), and $A_{ushape}$ (right) with CG,
    linesearch, and column locking.}
    \label{fig:converge_rate_cg}
\end{figure}

Figure~\ref{fig:converge_rate_fixed_stepsize} shows the convergence
behaviors of \TOMOne{} applied to three random matrices. Nonlinear
convergence is observed in all three figures for the first few iterations.
Linear convergence is then observed until convergence. This agrees with
our analysis.

\begin{table}[ht]
    \centering
    \begin{tabular}{lcccccc}
        \toprule
        & & \multicolumn{5}{c}{Convergence rate} \\
        \cmidrule{3-7}
        Matrix & & $\lambda_1$ & $\lambda_2$ & $\lambda_3$ & $\lambda_4$
        & $\lambda_5$ \\
        \toprule
        \multirow{2}{*}{$A_{log}$} & Reference rate &
        $0.7952$ & $0.8976$ & $0.9488$ & $0.9744$ & $0.9872$ \\
        & Numerical rate &
        $0.7952$ & $0.8976$ & $0.9488$ & $0.9744$ & $0.9872$ \\
        \bottomrule
    \end{tabular}
    \caption{ Local convergence rate of \TOMOne{} applying to
    $A_{log}$.} \label{tab:convergence_rate}
\end{table}

In Figure~\ref{fig:converge_rate_fixed_stepsize} left, all curves are
parallel to each other, and their convergence rates are the same, which
agrees with our analysis since $A_{uni}$ has all equal eigengaps. In
Figure~\ref{fig:converge_rate_fixed_stepsize} middle, curves have
different slopes and hence different convergence rates. Here we provide a
quantitative comparison of the convergence rates for $A_{log}$ in
Table~\ref{tab:convergence_rate}. We fit the slopes of curves and use them
as numerical rates, whereas reference rates are computed according
Theorem~\ref{thm:local-conv-obj1}. Table~\ref{tab:convergence_rate} shows
that numerical rates agree with reference rates up to four digits. Hence
we claim that the rate in Theorem~\ref{thm:local-conv-obj1} is tight. In
Figure~\ref{fig:converge_rate_fixed_stepsize} right, the first four curves
have convergence rates that agree with our theoretical results, whereas
the last one converges faster than expected. Its convergence rate is
theoretically upper bounded by that of previous columns but numerically is
faster. Through these numerical results, we claim that our theoretical
analysis of the local convergence rate provides a tight upper bound for
practice.

In Section~\ref{sec:localconv}, we prove the local convergence column by
column, \ie, the convergence of later column is proved if all earlier
columns are close to their stable fixed points. An interesting numerical
observation from Figure~\ref{fig:converge_rate_fixed_stepsize} is that the
linear convergences of columns may have some overlapping iterations, \eg,
there are a lot of iterations that all curves converge linearly parallelly
for matrix $A_{uni}$. Such overlapping leads to faster convergence for the
overall algorithm.

\subsubsection{Accelerating Techniques}
\label{sec:numres_techniques}

In this section, we investigate the accelerating techniques. The
convergence behaviors of \TOMOne{} are included in
Figure~\ref{fig:converge_rate_cg}. Overall, the convergence of \TOMOne{}
with accelerating techniques are much faster than that of vanilla
\TOMOne{}. Next, we provide more quantitative comparisons for column
locking and momentum accelerations for both \TOMOne{} and \TOMTwo{}.

Specifically, in this section, the tolerance $\epsilon$ used for stopping
criteria and column locking is $10^{-8}$. And each experiment is repeated
$500$ times, with random matrices and initial values. For the number of
iterations (Iter Num) and the number of matrix-vector multiplications
(Mat-Vec Num), we report the mean, max, and min among $500$ random tests.
In all tests, the linesearch is always enabled.

\begin{table}[ht]
    \centering
    \begin{tabular}{lcccccc}
        \toprule \multirow{2}{*}{Method} &
        \multicolumn{3}{c}{Iter Num} &
        \multicolumn{3}{c}{Mat-Vec Num} \\
        \cmidrule(lr){2-4} \cmidrule(lr){5-7}
        & Mean & Max & Min & Mean & Max & Min \\
        \toprule
        \TOMOne{} + CG +locking &
        642.2 & 800 & 554 & 4990.2 & 6905 & 4353 \\
        \TOMOne{} + CG &
        643.1 & 832 & 518 & 6431.4 & 8320 & 5180 \\
        \bottomrule
    \end{tabular}
    \caption{Performance comparison of \TOMOne{} applied to $A_{uni}$ with
    and without column locking. CG and exact linesearch are enabled.}
    \label{tab:comp_locking}
\end{table}

First, we show the advantage of column locking.
Table~\ref{tab:comp_locking} list the results for \TOMOne{} applied to
$A_{uni}$ with and without column locking. We observe that the numbers of
iterations remain the same with and without the column locking. However,
the number of matrix-vector multiplication is significantly reduced with
column locking, and hence the computational cost is reduced. Similar
results are observed for \TOMOne{} on other matrices and \TOMTwo{} as
well. We omit those results for the sack of brevity.

Then we explore the advantages of momentum and CG techniques. For
algorithms with vanilla momentum acceleration, the coefficient are chosen
as $\beta = 0.9$ for \eqref{eq:obj1} and $\beta = 0.95$ for
\eqref{eq:obj2}. Several different values of $\beta$ have been tested for
both objective functions, and we pick these $\beta$s for objective
functions with the fastest convergence.

\begin{table}[htp]
    \centering
    \begin{tabular}{lrrrrrrr}
        \toprule
        \multirow{2}{*}{\parbox{4.5em}{\centering Objective\\Function}} &
        \multirow{2}{*}{Method} & \multicolumn{3}{c}{Iter Num} &
        \multicolumn{3}{c}{Mat-Vec Num} \\
        \cmidrule(lr){3-5}
        \cmidrule(lr){6-8}
        & & Mean & Max & Min & Mean & Max & Min \\ 
        \toprule
        \multirow{6}{*}{\eqref{eq:obj1}}
        & \TOM+CG &  49.0 &  59 &  40 &  414.7 &   519 &  334 \\
        & OFM+CG  & 616.1 & 1881 & 333 & 6161.4 & 18810 & 3330 \\
        & \TOM+Momentum &   46.4 &   58 &  38 &   401.4 &   510 &  335 \\
        & OFM+Momentum  & 963.6 & 1468 & 614 & 9635.6 & 14680 & 6140 \\
        & \TOM+GD &    52.1 &    67 &    42 &    492.0 &    635 &    415 \\
        & OFM+GD  & 11460.7 & 17124 & 4591 & 114607.2 & 171240 & 75910 \\
        \midrule
        \multirow{6}{*}{\eqref{eq:obj2}}
        & \TOM+CG &  279.0 & 553 & 193 &  1071.0 & 1499 & 882 \\
        & OFM+CG  & 953.2 & 2500 & 550 & 9532.2 &  25000 & 5500 \\
        & \TOM+Momentum &  701.4 & 997 &  504 &  2217.0 &  2588 &  1840 \\
        & OFM+Momentum  & 1275.3 & 2033 & 738 & 12752.8 & 20330 & 7380 \\
        & \TOM+GD &  5150.7 & 9280 &  2663 & 12168.2 &  16500 &  7214 \\
        & OFM+GD  & 21222.2 & 30462 & 14156 & 212221.9 & 304620 & 141560 \\
        \bottomrule  
    \end{tabular}
    \caption{Performance comparison of \TOM{} and \OFM{} with and without
    momentum accelerating techniques for $A_{log}$. Here GD stands for the
    vanilla gradient descent method. Exact linesearch is enabled for all
    algorithms and column locking is enabled for \TOM{}.}
    \label{tab:comparison_log}
\end{table}


Numerical results are summarized in Table~\ref{tab:comparison_log} for
$A_{log}$. In all cases, \TOM{}s converge in less number of iterations and
less number of matrix-vector multiplications. There are two reasons behind
the results. First, the convergences of earlier columns in \TOM{} are
faster than that of the last column, whereas the convergences of all
columns in \OFM{} are the same as the last column in \TOM{}. Second, in
\TOM{}, different linesearch stepsizes are applied to different columns,
whereas OFM uses a single stepsize for all columns, which is impacted by
the smallest eigengap. Overall, the computational costs of \TOM{} and
\OFM{} depend on the eigengap distribution of the matrix. For
$A_{log}$-like matrices, \TOM{} outperforms \OFM{}. Further, algorithms
with CG converge faster or equally fast as their momentum accelerated
versions with carefully chosen parameter $\beta$s. Hence we recommend CG
as the momentum acceleration since it is hyper-parameter free.

Through all these tests, our best choice is to use \TOM{} with CG, exact
linesearch, and column locking. As in the later sections, this
configuration will be the default \TOM{}, and we will focus on the
sparsity of eigenvectors.

\subsection{Synthetic Density Functional Theory}
\label{sec:numres_dft}

In this section we perform \TOM{} on a synthetic example from DFT
computation. The example is a second order differential operator on the
domain $[0,1]$ with periodic boundary condition,
\begin{equation} \label{eq:dft_eg}
    H(x) = -\Delta + V(x),
\end{equation}
where $-\Delta$ is the Laplace operator denoting the kinetic term and
$V(x)$ is a local potential with four Gaussian potential wells,
\begin{equation} \label{eq:dft_potential}
    V(x) = -\sum_{i = 1}^4 \alpha_i e^{-\frac{(x -
    \ell_i)^2}{2\sigma^2}}.
\end{equation}
The centers of these wells locate at $\ell_i = \frac{2i-1}{8}$, the depths
of the wells are $\alpha_i = 850 + 50 \times \mathrm{mod}(i,4)$, and the
constant width of these wells is $\sigma = 0.1$. This second order
differential operator, \eqref{eq:dft_eg}, can be viewed as the linear
operator in a self-consistent field iteration in DFT computation,
simulating four different atoms located periodically on a line. In this
example, we are interested in computing the low-lying four eigenpairs. The
associated matrix is obtained via discretizing the problem on a uniform
grid with $n = 500$ points, where the Laplace operator is discretized
using the central difference scheme. In Figure~\ref{fig:dft_eigenvector}
left, we plot the four eigenvectors corresponding to smallest four
eigenvalues. Due to the localized potential and periodicity, the
eigenvectors associated with low-lying eigenvalues have localized
property, which means that these eigenvectors are sparse.

\begin{figure}[ht]
    \includegraphics[width=0.333\textwidth]{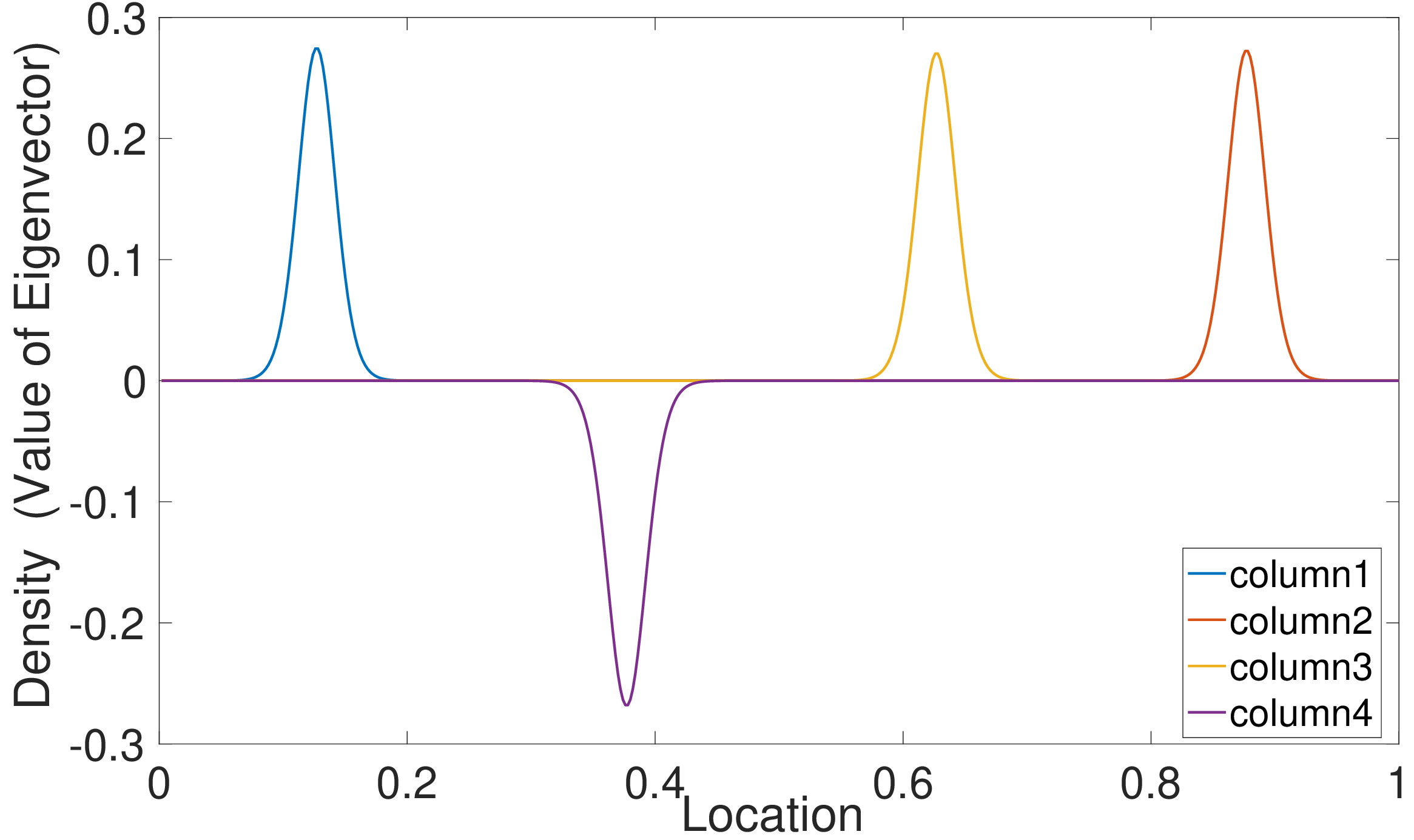}
    \includegraphics[width=0.333\textwidth]{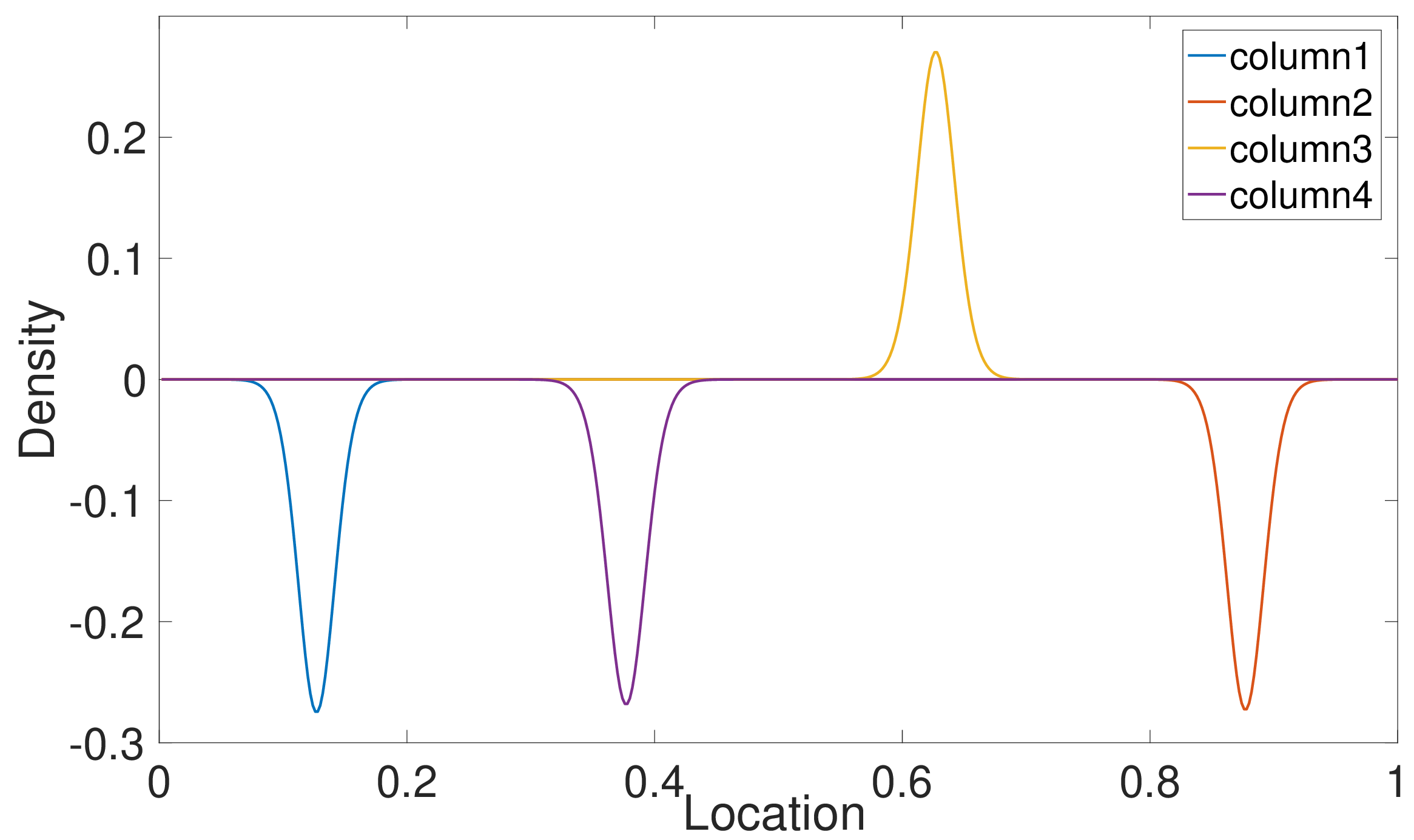}
    \includegraphics[width=0.333\textwidth]{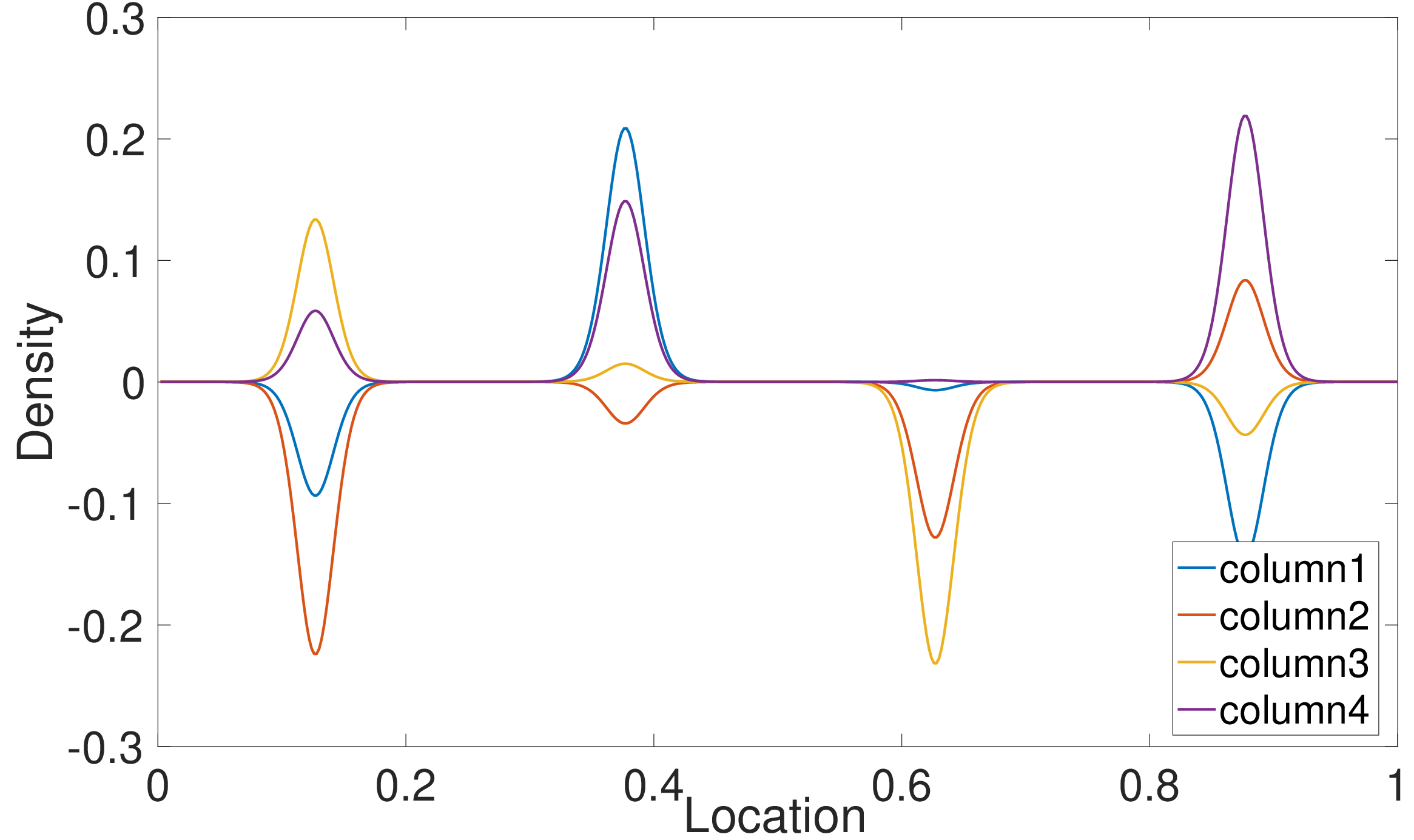}
    \caption{Left figure plots the ground truth eigenvectors associated
    with four low-lying eigenvalues of problem \eqref{eq:dft_eg}; middle
    figure plots scaled four convergent columns from \TOMOne{}; and right
    figure plots scaled four convergent columns from the \OFMOne{}.}
    \label{fig:dft_eigenvector}
\end{figure}

\begin{table}[ht]
    \centering
    \begin{tabular}{lccccc}
        \toprule
        Method & Iter Num & Mat-Vec Num & NNZ &
        $e_{vec}$ & $e_{val}$\\
        \midrule
        \TOM{}+CG & 567.5 & 5134.2 & 1328 & $8.26\times 10^{-8}$
        & $2.00\times 10^{-15}$\\
        OFM+CG    & 413.6 & 4135.7 & 4974.0 & --
        & $2.50\times 10^{-15}$\\
        \bottomrule
    \end{tabular}
    \caption{Performance comparison of \TOMOne{} and \OFMOne{} applied to
    \eqref{eq:dft_eg}.}
    \label{tab:numres_dft}
\end{table}

Numerical results are demonstrated in Figure~\ref{fig:dft_eigenvector} and
Table~\ref{tab:numres_dft}. The tolerance is $10^{-8}$, and each algorithm
is perform 100 times with random initial states.
Figure~\ref{fig:dft_eigenvector} middle plots the scaled four convergent
columns from \TOMOne{} and the right figure plots scaled four convergent
columns from the non-triangularized counterpart.
Table~\ref{tab:numres_dft} includes the number of iterations, the number
of matrix-vector multiplications, the number of nonzeros (NNZ), and the
accuracies. In Table~\ref{tab:numres_dft}, the NNZ is the number of
entries with absolute values greater than $10^{-5}$. The NNZ of the ground
truth eigenvectors is 1328. Since \OFMOne{} does not provide eigenvectors
without an extra orthogonalization step, the accuracy of eigenvectors is
not available.

According to Figure~\ref{fig:dft_eigenvector}, the convergent columns of
\TOMOne{} recover the eigenvectors up to a sign difference. While the
convergent columns of \OFMOne{} mix all four eigenvectors and have
nonzeros near all four Gaussian centers. Hence the sparsity of
eigenvectors is destroyed. Overall, \TOMOne{} achieves $75\%$ saving in
NNZ comparing to that of \OFMOne{}. Since the memory cost is a key
bottleneck in many DFT computations, such a saving is important.
Meanwhile, regarding the number of iterations and the number of
matrix-vector multiplications, \TOMOne{} is slightly more expensive than
\OFMOne{}. Hence there is a trade-off between time and space. If the
parallelizability of \TOM{} is further taken into account, then \TOM{}
would be a valuable alternative eigensolver for DFT.

\subsection{Full Configuration Interaction}
\label{sec:numres_fci}

This section solves the low-lying eigenpairs for a two-dimensional Hubbard
model under the FCI framework. The Hubbard model is widely used in
solid-state physics, which only considers the neighboring hopping and
on-site interaction. Under the FCI framework, the matrix size scales
factorially with respect to the problem size and the number of
electrons. The eigenvectors associated with low-lying eigenvalues are
sparse. FCI problems are the most important applications of \TOM{}.

The Hamiltonian operator in the second quantization notation is,
\begin{equation}
    \hat{H} = -t \sum_{\langle r,r'\rangle, \sigma}
    \hat{a}_{r,\sigma}^\dagger \hat{a}_{r', \sigma}
    + U \sum_{r} \hat{a}_{r,\uparrow}^\dagger \hat{a}_{r,\uparrow}
    \hat{a}_{r,\downarrow}^\dagger \hat{a}_{r,\downarrow}
\end{equation}
where $t$ is the hopping strength, $U$ is the interaction strength, $r,r'$
are lattice index, $\langle r, r'\rangle$ means that $r$ and $r'$ are
neighbors on the lattice, $\hat{a}_{r,\sigma}^\dagger$ and
$\hat{a}_{r,\sigma}$ denotes the creation and annihilation operator of an
electron with spin $\sigma$ on $r$. The matrix in this section is
generated from the Hubbard model in momentum space. The Fourier transform
of the creation and annihilation operator is $\hat{a}_{k,\sigma} =
\frac{1}{\sqrt{N^{site}}} \sum_r e^{\imath k\cdot r} \hat{a}_{r,\sigma}$,
where $k$ is the wave number and $N^{site}$ is the number of lattice
sites. The Hamiltonian operator in momentum space is,
\begin{equation} \label{eq:ham-fci}
    \hat{H} = t \sum_{k, \sigma} -2 (\cos{k_1} + \cos{k_2})
    \hat{a}_{k,\sigma}^\dagger \hat{a}_{k, \sigma} +
    \frac{U}{N^{site}} \sum_{k,p,q} \hat{a}_{p-q,\uparrow}^\dagger
    \hat{a}_{k+q,\downarrow}^\dagger \hat{a}_{k,\downarrow}
    \hat{a}_{p,\uparrow}
\end{equation}
for $k = (k_1,k_2)$.

We adopt a 2D Hubbard model on a lattice of size $4 \times 4$ with $8$
electrons (4 spin-up and 4 spin-down). The strength of hopping and
interaction are $t = 1$ and $U = 0.25 N^{site}$ respectively. The FCI
matrix has diagonal entries between $-20$ and $20$ and off-diagonal
entries being $\pm 0.25$. The matrix size is about $(2\cdot 10^{5})\times
(2\cdot 10^{5})$. We compute the smallest $p = 10$ eigenpairs. \TOMOne{}
and \OFMOne{} are applied to address this problem. The tolerance is
$10^{-10}$. For each algorithm, we perform 100 times with random initial
states. The mean of the number of iterations, the number of matrix-vector
multiplications, NNZ, and accuracies are reported in Table~\ref{tab:fci}.
Similarly, the NNZ is the number of entries with a magnitude greater than
$10^{-5}$.

\begin{table}[ht]  
    \centering
    \begin{tabular}{lrrrcc}
        \toprule
        Method & Iter Num & Mat-Vec Num & NNZ &
        $e_{vec}$ & $e_{val}$\\
        \midrule
        \TOM{}+CG & 1253.0 & 7708.6 & $1.115 \times 10^{6} $
        & $3.308 \times 10^{-8}$ & $5.597 \times 10^{-12}$\\
        OFM+CG    & 1381.7 & 13817.2 & $1.508 \times 10^{6}$
        & --                    & $5.197 \times 10^{-12}$\\
        \bottomrule
    \end{tabular}
    \caption{Performance comparison of \TOMOne{} and \OFMOne{} on
    \eqref{eq:ham-fci}}
    \label{tab:fci}
\end{table}

According to Table~\ref{tab:fci}, \TOMOne{} requires less number of
iterations and matrix-vector multiplications than \OFMOne{}. In FCI
problems, the number of matrix-vector multiplications is proportional to
the actual runtime. Hence we expect that \TOMOne{} would achieve better
runtime than \OFMOne{} on FCI problems. Notice that the multiplicity of
some eigenvalues in our FCI matrix is not one. Hence the stable fixed
points are subspaces. NNZ for \TOMOne{} varies over 100 executions and
Table~\ref{tab:fci} reports its mean. On average, \TOMOne{} achieves
better sparsity comparing to \OFMOne{}. Through our numerical results,
\TOMOne{} outperforms \OFMOne{} on the FCI problem.

\begin{remark}

Solving practical FCI problems is the major target in designing \TOM{}.
In FCI problems, low-lying eigenvalues and the associated eigenvectors are
computed as the ground state and low-lying excited states. Almost all
traditional eigensolvers are not applicable to FCI problems. \OFM{} with
coordinate-wise descent method is an option to obtain the sparse
eigenvectors. While the arbitrary rotation would significantly increase
the memory cost. Hence, we design \TOM{} converging to the sparse
eigenvectors directly. According to our numerical result of the FCI
problem, \TOM{} outperforms \OFM{} and is a more promising method to
address FCI problems. This paper is the first step toward computing the
FCI excited states. Coupling \TOM{} together with a parallelized
coordinate-wise descent method, we would be able to address FCI problems
for transition metals of interest. 

\end{remark}

\section{Conclusion and Discussion}
\label{sec:conclusion}

In this work, we introduce a novel \TOM{} for solving extreme eigenvalue
problems. Using \TOM{}, the eigenpairs are directly solved via
orthogonalization-free iterative methods, where the orthogonalization-free
feature is crucial for large-scale eigenvalue problems with sparse
eigenvectors. Two specific algorithms, namely \TOMOne{} and \TOMTwo{}, are
proposed for \eqref{eq:obj1} and \eqref{eq:obj2}. Global convergences are
guaranteed for almost all initial states~\cite{Gao2021}. Locally, we prove
that, in neighbors of stable fixed points, \TOMOne{} converges linearly.
The convergence proof can be adapted to show the linear convergence of
\TOMTwo{} as well. Although the proposed algorithms are different from
general gradient-based algorithms, acceleration techniques, including
momentum, linesearch, and column locking, still work effectively.
According to numerical results on both synthetic examples and the example
from practice, \TOMOne{} and \TOMTwo{} converge efficiently and obtain the
sparse eigenvectors without any orthogonalization step.

There are many future directions. As has been mentioned before, \TOM{} is
applicable to many other objective functions beyond \eqref{eq:obj1} and
\eqref{eq:obj2}. We would like to apply \TOM{} to other objective
functions and obtain powerful algorithms. Moreover, we claim the advantage
of \TOM{} in keeping sparsity towards convergent. It is an interesting
future direction to explore truncation techniques and coordinate-wise
methods so that the sparsity is preserved throughout iterations. The
application to FCI problems would be of great interest to many other
communities, including computational physics, chemistry, and material
science, etc. In addition to the above two directions,
orthogonalization-free algorithms are friendly to parallel computing.
Hence the parallelization of these proposed algorithms is another future
direction.


\medskip
\noindent

{\bf Acknowledgments.} The authors thank Jianfeng Lu and Zhe Wang for
helpful discussions. YL is supported in part by the US Department of
Energy via grant de-sc0019449. WG is supported in part by National Science
Foundation of China under Grant No. 11690013, U1811461.

\bibliographystyle{apalike} \bibliography{library}

\appendix

\section{Proof of Theorem~\ref{thm:stationarypt-obj2}}
\label{app:stationarypt-obj2}

\begin{proof}[\bf Proof of Theorem~\ref{thm:stationarypt-obj2}]

All fixed points of \eqref{eq:triofm-obj2} satisfy $g_2(X) = 0$. We first
analyze the fixed points for a single column case and then complete the
proof by induction. Notations used in this proof are the same as that in
the proof of Theorem~\ref{thm:stationarypt-obj1}.

We denote the single column $X$ as $x$. Obviously, when $x = 0$, we have
$g_2(x) = 0$. Now, consider the nontrivial case $x \neq 0$. The equality
$g_2(x) = 0$ can be expanded as,
\begin{equation} \label{eq:eigzeroobj2}
    \left( (2 - x^\top x)A - x^\top A x I \right) x = 0.
\end{equation}
According to \eqref{eq:eigzeroobj2}, for nonzero $x$, the matrix $B = (2 -
x^\top x)A - x^\top A x I$ must has a zero eigenvalue and $x$ lies in its
corresponding eigenspace. When $x^\top x = 2$, the matrix $B = x^\top A x
I$ does not have zero eigenvalue due to the negativity assumption on $A$.
Hence $x$ is parallel to one of $A$'s eigenvector, \ie, $Ax = \lambda x$.
Substituting this into \eqref{eq:eigzeroobj2}, we obtain,
\begin{equation}
    2 (1 - x^\top x) \lambda x = 0.
\end{equation}
Since $\lambda < 0$ and $x \neq 0$, we have $x^\top x = 1$. Hence we
conclude that for $g_2(x) = 0$, $x$ is either a zero vector or an
eigenvector of $A$.

Now we consider multicolumn case. The first column of $g_2(X) = 0$ is the
same as \eqref{eq:eigzeroobj2}. Hence $X_1 = U P_1 S_1$.

Assume the first $i$ columns of $X$ obey $X_i = U P_i S_i$. Then the
$(i+1)$-th column of $g_2(X) = 0$ is
\begin{equation} \label{eq:eigzeroiobj2}
    2A x_{i+1} - Ax_{i+1} x_{i+1}^\top x_{i+1} - x_{i+1}
    x_{i+1}^\top A x_{i+1} - AX_i X_i^\top x_{i+1} - X_i X_i^\top
    Ax_{i+1} = 0.
\end{equation}
Obviously, if $x_{i+1} = 0$, then \eqref{eq:eigzeroiobj2} holds. When
$x_{i+1} \neq 0$, we left multiply \eqref{eq:eigzeroiobj2} with
$X_i^\top$, adopt the commuting property of diagonal matrices, and obtain,
\begin{equation} \label{eq:eigzeroxtiobj2}
    \begin{split}
        0 = & S_i P_i^\top \left( 2\Lambda - x_{i+1}^\top
        x_{i+1} \Lambda - x_{i+1}^\top A x_{i+1} I - \Lambda P_i
        P_i^\top - \Lambda \right) U^\top x_{i+1} \\
        = & - S_i P_i^\top \left( x_{i+1}^\top x_{i+1} \Lambda
        + x_{i+1}^\top A x_{i+1} I \right) U^\top x_{i+1} \\
    \end{split}
\end{equation}
where the second equality adopts the fact that $P_i^\top \Lambda P_i
P_i^\top  = P_i^\top \Lambda$.  Due to the negativity of $A$, we notice
that $x_{i+1}^\top x_{i+1} \Lambda + x_{i+1}^\top A x_{i+1} I$ is a
diagonal matrix with strictly negative diagonal entries. Hence the
equality \eqref{eq:eigzeroxtiobj2} is equivalent to
\begin{equation} \label{eq:eigzeroshortobj2}
    S_i P_i^\top U^\top x_{i+1} = 0.
\end{equation}
As long as \eqref{eq:eigzeroshortobj2} holds, we have $X_i^\top x_{i+1} =
0$ and $X_i^\top A x_{i+1} = 0$. Therefore, solving
\eqref{eq:eigzeroiobj2} can be addressed via solving
\begin{equation}
    2A x_{i+1} - Ax_{i+1} x_{i+1}^\top x_{i+1} - x_{i+1}
    x_{i+1}^\top A x_{i+1} = 0.
\end{equation}
Hence $x_{i+1}$ satisfies \eqref{eq:eigzeroshortobj2}. Combining the
solution of the single column case \eqref{eq:eigzeroobj2} and the
constraint \eqref{eq:eigzeroshortobj2}, we conclude that $X_{i+1}$ is of
the form $U P_{i+1} S_{i+1}$.

The stabilities of fixed points should also be analyzed through the
spectrum properties of their Jacobian matrices. The Jacobian matrix $\Diff
g_2(X)$, again, can be written as a $p$-by-$p$ block matrix. And using the
similar argument as in the proof of Theorem~\ref{thm:stationarypt-obj1},
$\Diff g_2(X) = \Diff G$ is a block upper triangular matrix whose
spectrum is determined by the spectrum of its diagonal blocks. Through a
multivariable calculus, we obtain the expression for $J_{ii}$ as,
\begin{equation} \label{eq:Jiiobj2}
    J_{ii} = 2A - A X_i X_i^\top - X_i X_i^\top A - A x_i x_i^\top -
    x_i^\top x_i A - x_i^\top A x_i I - x_i x_i^\top A.
\end{equation}

We first show the stability of the fixed points of form $X = U_p D$.
Substituting these points into \eqref{eq:Jiiobj2}, we have,
\begin{equation} 
    J_{ii} = A - 2 U_i \Lambda_i U_i^\top - 2 \lambda_i u_i u_i^\top
     - \lambda_i I.
\end{equation}
Since $\lambda_i$ is smaller than all eigenvalues of $A - U_i \Lambda_i
U_i^\top$, $A - U_i \Lambda_i U_i^\top - \lambda_i I$ is strictly positive
definite. The rest part of \eqref{eq:Jiiobj2} is, obviously, positive
definite. Hence $J_{ii}$ is strictly positive definite for all $i = 1, 2,
\dots, p$ and fixed points of the form $X = U_p D$ are stable fixed
points.

Next we show the rest fixed points are not stable. For a fixed point $X$,
we denote the first index $s$ such that $x_s^\top u_s = 0$. Then we
estimate $u_s^\top J_{ss} u_s$ as,
\begin{equation}
    u_s^\top J_{ii} u_s = 2\lambda_s - x_s^\top x_s \lambda_s -
    x_s^\top A x_s < 0,
\end{equation}
since $x_s^\top x_s \leq 1$ and $A$ is negative definite. Therefore, the
rest fixed points are not stable.

\end{proof}

\end{document}